\documentclass[11 pt]{article}

\usepackage[english]{babel}
\usepackage{tabu}
\usepackage{amssymb}
\usepackage{amsthm}
\usepackage{amsmath}
\usepackage{xcolor}
\usepackage[T1]{fontenc}
\usepackage{dsfont}
\usepackage{geometry} % margins
\usepackage[final]{graphicx}
\usepackage{tikz}
\usepackage{caption}
\usepackage{stmaryrd}
\def\prob#1{\mathbb{P}\left[ #1 \right]}
\def\zn{\mathbb{Z}_n}
\def\D{D}
\def\b#1{\mathbf{b}_{#1}}
\def\auto{\mathcal{A}_n(\mathbf{b})}

\def\RU{\mathcal{E}_{\mathrm{\small row}} (\alpha)}
\def\cRU{\mathcal{E}_{\mathrm{\small row}}^c (\alpha)}

\def\cRUstar{\mathcal{E}_{\mathrm{\small row}}^c (\alpha^{\star})}
\def\RD{\mathcal{E}_{\mathrm{\small zero}} (\beta)}
\def\cRD{\mathcal{E}_{\mathrm{\small zero}}^c (\beta)}

\def\cRDstar{\mathcal{E}_{\mathrm{\small zero}}^c (\beta^{\star})}
\def\nmed{\left\lfloor\frac{n}{2}\right\rfloor}

\def\db#1#2{\big|b_{#1}-b_{(#1+#2)_n}\big|_n}
\def\A#1{\mathcal{R}^{\alpha}_{#1}(\mathbf{b})}

\def\auto{\mathcal{A}_n}
\def\b#1{\mathbf{b}_{#1}}
\def\var#1{\mathbb{V}\left[#1\right]}

\def\dist#1{\big| #1 \big|_n}

\def\E#1{\mathbb{E}\left[ #1 \right]}
\def\cov#1{\mbox{Cov} \left[ #1 \right]}

\def\P#1{P_{#1}(n)}\def\E#1{\mathbb{E}\left[ #1 \right]}
\def\indi#1{\mathds{1}\left\{ #1 \right\}}
\def\Z{\mathcal{Z}}
\newcommand{\Mod}[1]{\ (\mathrm{mod}\ #1)}
\def\med{\frac{n}{2}}
\def\Tb{T_{\b{}}}
\def\largo{1.0}
\def\porcion{0.1}
\def\sumOne{\substack{1 \leq i,r\leq \nmed \\0 \leq j ,  j', s, s' \leq n-1\\
j < j' ;\, s < s'	
\\(i,j,j') \neq (r,s,s')}}
\def\sumOneEven{\substack{1 \leq i,r \leq \med \\ 0\leq j, j', s, s'\leq n-1\\
j<j';s<s'\\
j \not \equiv_{ \text{\tiny$\frac{n}{2}$}} j' ;\, s \not \equiv_{\text{\tiny$\frac{n}{2}$}} s'\\ (i,j,j') \neq (r,s,s')}}
\newcommand{\autor}[3]{#1}
\def\Cerny{\v{C}ern\`y }
\def\A{\mathcal{A}}

\newtheorem{thm}{Theorem}
\newtheorem{lemma}[thm]{Lemma}
\newtheorem{claim}[thm]{Claim}

\newtheorem{prop}[thm]{Proposition}

\theoremstyle{remark}
\newtheorem{remark}[thm]{Remark}
\newtheorem*{remark*}{Remark}
\newtheorem*{defi}{Definition}
\newtheorem{conj}[thm]{Conjecture}

\title{ Circular automata synchronize with high probability}

\author{
\autor{Christoph Aistleitner\thanks{Institute of Analysis and Number Theory, TUGraz, Austria. aistleitner@math.tugraz.at}}{}{}
\and
\autor{Daniele D'Angeli\thanks{Universit\`a Niccol\`o Cusano, Via don Gnocchi Roma, Italia. daniele.dangeli@unicusano.it}}{}{}
\and
\autor{Abraham Gutierrez\thanks{Institute of Discrete Mathematics, TUGraz, Austria. \{a.gutierrez, rosenmann\}@math.tugraz.at}}{}{}
\and
\autor{Emanuele Rodaro\thanks{Department of Mathematics, Politecnico di Milano, Italia. emanuele.rodaro@polimi.it}}{}{}
\and
\autor{Amnon Rosenmann$^{\ddag}$}{}{}
}
\begin{document}
\date{}

	\maketitle	
	\begin{abstract}
		 In this paper we prove that a uniformly distributed random circular automaton $\auto$ of order $n$ synchronizes with high probability (w.h.p.). More precisely, we prove that
		$$
		\prob{\auto \text{ synchronizes}} = 1- O\left(\frac{1}{n}\right).
		$$
		The main idea of the proof is to translate the synchronization problem into a problem concerning properties of a random matrix; these properties are then established with high probability by a careful analysis of the stochastic dependence structure among the random entries of the matrix. Additionally, we provide an upper bound for the probability of synchronization of circular automata in terms of chromatic polynomials of circulant graphs.
	\end{abstract}
\section{Introduction}
 A \textit{complete deterministic finite automaton} (DFA) is a tuple $\mathcal{A} =(Q,L)$, where  $Q:=\{q_1,q_2,\ldots,q_n\}$ is a finite set of \textit{states} and $L := \{\mathbf{a_1},\mathbf{a_2},\ldots,\mathbf{a_k}\}$ is a finite set of mappings $\mathbf{a_i}:Q\rightarrow Q$, where $\mathbf{a}(q)=q'$ is also written as $q \mathbf{a} = q'$, $q,q' \in Q$, $\mathbf{a} \in L$. 
 The number of states $n$ is the \textit{order} of $\mathcal{A}$.
 Each $\mathbf{a_i}$ is called a \textit{letter} and a sequence $\mathbf{w} = \mathbf{a_{i_1}a_{i_2}}\ldots \mathbf{a_{i_r}} \in L^*$ is a \textit{word} of \textit{length} $r$.
 The action of $L$ on $Q$ naturally extends to an action of $L^*$ on $Q$, defined recursively by $ q (\mathbf{aw}) = (q \mathbf{a}) \mathbf{w}$, $q\in Q$, $\mathbf{a}\in L$, $\mathbf{w} \in L^*$. This action further extends to an action of $L^*$ on subsets of $Q$ by $\{q_{i_1},q_{i_2},\ldots,q_{i_k}\}\mathbf{w}=\{q_{i_1}\mathbf{w},q_{i_2}\mathbf{w},\ldots,q_{i_k}\mathbf{w}\}$.
 We say that the subset $S=\{q_{i_1},q_{i_2},\ldots,q_{i_k}\}\subseteq Q$ \textit{synchronizes} if there exists a word $\mathbf{w} \in L^*$ such that $q_{i_1}\mathbf{w} = q_{i_2}\mathbf{w} = \ldots = q_{i_k}\mathbf{w}$ (equivalently, we say that $\mathbf{w}$ synchronizes $S$).
 If the set $Q$ synchronizes then we say that $\mathcal{A}(Q,L)$ \textit{synchronizes} (or that it is a synchronizing automaton).
 A word $\mathbf{w} \in L^*$ that synchronizes $Q$ is called a \textit{synchronizing} (or \textit{reset}) word of $\A$.
 
 The following simple criterion for synchronization is well known and plays a crucial role throughout the paper:
\begin{claim}\label{claim:SynchroEquivalence}
	$\mathcal{A}= (Q,L)$ synchronizes $\iff$ every pair of states $q,q'\in Q$ synchronizes.
\end{claim}
\begin{proof}
	It is clear that if $Q$ synchronizes by a reset word $\mathbf{w}$ then $\mathbf{w}$ synchronizes every pair of states of $Q$.
	Conversely, a reset word for $Q$ can be formed by concatenating words $w_i$ that synchronize pairs of states until we end up with a single state.
%	Conversely, let $N$ be a smallest subset among all the reachable subsets $\{Qw: w\mbox{ is a word on } L\}$, if $|N|>1$, then by taking a pair $q,q'$ of distinct states we may collapse them using a suitable word $u$ so that $|Nu|<|N|$, against the minimality of $N$.
\end{proof}
The synchronization property may be described in terms of the graph representation of $\mathcal{A}$. The set $Q$ of states comprises the vertices of the graph and for each pair of states $q,q'$ and a letter $\mathbf{a} \in L$ such that $q\mathbf{a}=q'$ there is an arrow 
$(q,q')_{\mathbf{a}}$ labeled with $\mathbf{a} \in L$ and connecting $q$ to $q'$.
% can be seen as a directed graph $D(\mathcal{A})$ with vertex set $Q$ and arrows labeled by $L$ where the arrow $(q,q')_{a_i}$ with label $a_i$ belongs to $D(\mathcal{A})$ if $a_i(q) = q'$.
Each $q\in Q$ and $w = \mathbf{a_{i_1}a_{i_2}}\ldots \mathbf{a_{i_k}} \in L^*$ defines a directed path $$
%\gamma(q,w) : = q,qa_{i_1},qa_{i_1}a_{i_2},\ldots, qa_{i_1}a_{i_2}\cdots a_{i_k}
\gamma(q,\mathbf{w}) : = ((q,q_{i_1})_{\mathbf{a_{i_1}}}, (q_{i_1},q_{i_2})_{\mathbf{a_{i_2}}}, \ldots, (q_{i_{k-1}},q')_{\mathbf{a_{i_{k}}}})
$$
that begins in $q$ and ends in $q'=q\mathbf{w}$. $\mathcal{A}$ then synchronizes if and only if there is a word $\mathbf{w}$, such that the paths $\{\gamma(q,\mathbf{w}) : q \in Q\}$ have a common endpoint $q'$, that is, the word $\mathbf{w}$ acts on $Q$ as the constant mapping.\\

Synchronizing automata have been intensely studied by theoretical computer scientists as well as pure mathematicians since the 1960's; see \cite{volkov2008synchronizing} for a detailed introduction on synchronization of automata. A driving force in this research field is the \Cerny conjecture.
\begin{conj}[The \Cerny conjecture]
	A synchronizing automaton $\mathcal{A}$ of order $n$ has a shortest synchronizing word of length at most $(n-1)^2$.
\end{conj}

 The bound in the \Cerny conjecture is tight: in \cite{cerny1964poznamka} \Cerny provided a series of synchronizing circular automata $C_2,C_3,\ldots$, such that $C_n$ has order $n$ and its shortest synchronizing word is of size exactly $(n-1)^2$ (see Fig. \ref{cerny automaton}). Furthermore, the \Cerny series of circular automata  $C_2,C_3,\ldots$ is the only known infinite series of automata whose shortest synchronizing words are of length $(n-1)^2$ \cite{VolkovExtremeAutomata}.
\begin{figure}[ht]
	\begin{center}
		\begin{tikzpicture}[>=latex, shorten >=1pt, shorten <=1pt]
			\tikzstyle{normal_node}= [draw,circle,inner sep=0pt,thick,minimum size=0.8cm]
			\draw (0,0) node [normal_node] (n-3) {$n{-}3$};
			\draw (2,0) node [normal_node] (3) {$3$};
			\draw (3.25,1.56) node [normal_node] (2) {$2$};
			\draw (2.8,3.51) node [normal_node] (1) {$1$};
			\draw (1,4.38) node [normal_node] (0) {$0$};
			\draw (-0.8,3.51) node [normal_node] (n-1) {$n{-}1$};
			\draw (-1.25,1.56) node [normal_node] (n-2) {$n{-}2$};
			\path (0) edge [->, bend left=10, red, thick] node[above] {$a$} (1)
	 		(1) edge [->, bend left=10, red, thick] node[right] {$a$} (2)
	 		(2) edge [->, bend left=10, red, thick] node[right] {$a$} (3)
 	 		(3) edge [->, bend left=10, red, dashed, thick] node[below] {$a$} (n-3)
 	 		(n-3) edge [->, bend left=10, red, thick] node[left] {$a$} (n-2)
 	 		(n-2) edge [->, bend left=10, red, thick] node[left] {$a$} (n-1)
 	 		(n-1) edge [->, bend left=10, red, thick] node[above] {$a$} (0)
 	 		(n-1) edge [->, bend right=10, blue, thick] node[below] {$b$} (0)
	 		(0) edge[->, out=110, in=70, distance=1.1cm, blue, thick] node[above] {$b$} (0)
	 		(1) edge[->, out=60, in=20, distance=1.1cm, blue, thick] node[right] {$b$} (1)
	 		(2) edge[->, out=10, in=330, distance=1.1cm, blue, thick] node[right] {$b$} (2)
	 		(3) edge[->, out=315, in=275, distance=1.1cm, blue, thick] node[below] {$b$} (3)
	 		(n-3) edge[->, out=265, in=225, distance=1.1cm, blue, thick] node[below] {$b$} (n-3)
	 		(n-2) edge[->, out=210, in=170, distance=1.1cm, blue, thick] node[left] {$b$} (n-2);
		\end{tikzpicture}
		\caption{The automaton $C_n$}
	\label{cerny automaton}
	\end{center}
\end{figure}
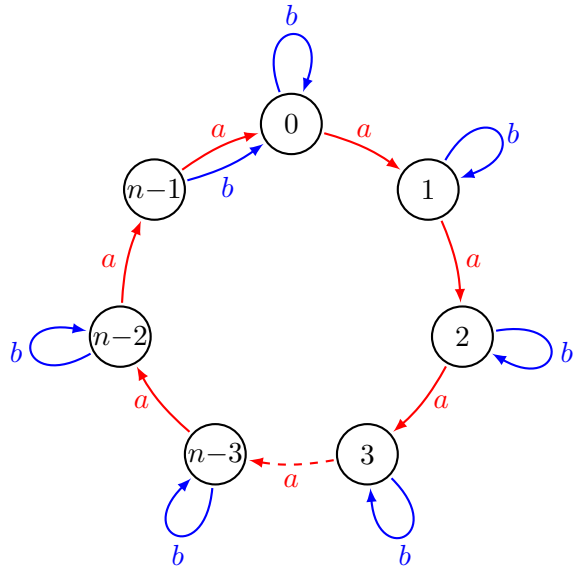
 The  best known general upper bounds for the size of shortest synchronizing words of an automaton with $n$ states are of order $O(n^3)$ \cite{pin1983CernyBound}\cite{szykula2017improvingCernyBound}\cite{Shitov2019}. Nevertheless, there are many classes of automata for which the \Cerny conjecture has been established (see \cite{volkov2008synchronizing} for examples).

In last decade probabilistic approaches to the synchronization problem have been developed. Typical questions in this setting are: let $\A(\{0,1,\ldots,n-1\},L)$ be a uniformly chosen DFA with $k$ letters on a certain probability space, is it true that with high probability the automaton $\A(\{0,1,\ldots,n-1\},L)$ is synchronizing? Does the \v{C}ern\`y conjecture hold with high probability? Here we give a (non-comprehensive) list of recent achievements in this probabilistic setting:
\begin{enumerate}
	\item[$\bullet$] In \cite{skvortsov20100LargeAlphabetSynchronization} the authors study random automata $\A$ where the number of letters $k$ grow together with $n$. In particular, they prove that $\A$ synchronizes w.h.p. when $k(n)$ grows fast enough;
	\item [$\bullet$] In \cite{Berlikov2013Synchronization} the author proves that $\prob{\A \text{ synchronizes}} = 1-O(n^{-k/2})$, for arbitrary $k\geq2$, and $\prob{\A \text{ synchronizes}} = 1-\Theta(1/n)$ for $k=2$;
	\item[$\bullet$] In \cite{nicaud2014fast} the author proves that $\A$ admits w.h.p. a synchronizing word of length $O(n\log^3 n)$ for arbitrary $k\geq 2;$
	\item[$\bullet$] In \cite{Nicaud2018AlmostGroupAutomata} the authors prove that if $\A$ is uniformly chosen among the strongly-connected almost-group automata then $\A$ synchronizes with probability $1-\Theta((2^{k-1}-1)n^{-2(k-1)})$ for arbitrary $k\geq 2$.
\end{enumerate}

Since the sequence of circular automata $C_n$ depicted in Fig. \ref{cerny automaton} is the only known infinite series of synchronizing automata reaching \Cerny's bound $(n-1)^2$, one might suspect that the class of circular automata is somehow difficult to synchronize. However, as we show in the present paper, it turns out that a random circular automaton is synchronizing with high probability.

The rest of the paper is organized as follows: in Section~\ref{sec:MainResult} we present the main result together with its proof and the statement of the two key lemmas for the proof. In Section~\ref{sec:DependenceStructureTb} we study the dependence structure among the entries of the random matrix used in the proof of the main result; the result obtained in this section is crucial for the proof of the key lemmas. In Section~\ref{sec:ProofLemmaRows} we prove the first lemma while in Section~\ref{sec:ProofLemmaZeros} we prove the second one. In Section~\ref{sec:ChromaticConnections} we present some interesting connections between synchronization of circular automata and chromatic polynomials of circulant graphs. Finally, in Section~\ref{sec:FutureWork} we present some possible directions towards generalizing and improving the results presented in this paper.

%****************************************************
\section{Main result} \label{sec:MainResult}
Let $n$ be a positive integer. An automaton $\mathcal{A}(\zn,L)$, where $\zn := \{0,1,\ldots,n-1\}$ is the set of states, is called a \textit{circular automaton} if $L$ contains a permutation that decomposes in exactly one cycle.
Let $(i)_n := i \mod  n$. Let $\mathcal{M}_n$ denote the set of all mappings from $\mathbb{Z}_n$ to itself, and let $\mathbb{P}$ denote the uniform probability measure on $\mathcal{M}_n$. We will write the elements of $\mathcal{M}_n$ as vectors by identifying the mapping $\mathbf{b}(i) = b_i, ~i=0,\dots, n-1$ with the vector $\mathbf{b}=(b_0, \dots, b_{n-1})$.

In what follows, we denote by $\mathcal{A}_n (\mathbf{b}) := (\zn,\{\mathbf{a},\mathbf{b}\})$ a circular automaton of order $n\in\mathbb{N}$, with $\mathbf{a}:\zn \rightarrow \zn$ being the circular right shift permutation $a(i) = (i+1)_n$ and $\mathbf{b} := (b_{0},...,b_{n-1})$ being an element of $\mathcal{M}_n$. We will understand that $\mathbf{b}$ is ``randomly'' chosen from $\mathcal{M}_n$ according to the uniform probability measure $\mathbb{P}$, making $\mathcal{A}_n (\mathbf{b})$ a random circular automaton.

It follows from work of Perrin \cite{perrin1977codes} that a circular automaton $\A(Q,L)$ of prime order synchronizes if and only if $L$ contains a non-permutation. Pin \cite{MR520853} proved with combinatorial methods that a circular automaton $\A(Q,L)$ of prime order which has a letter of rank $\frac{n-1}{2} \leq  k\leq n$ has a minimal word of size at most $(n-k)^2$. For the probability of synchronization of $\A_p(\mathbf{b})$ a very precise result is known.

\begin{thm}[\cite{perrin1977codes}\cite{MR520853}]\label{thm:Perrin}
	  Let $p$ be a prime number. Then
	  %where $\mathbb{Z}_p := \{0,1,\ldots,p-1\}$, where $a:\mathbb{Z}_p\rightarrow \mathbb{Z}_p$ is the mapping $a(i)\equiv_p i+1$, and where $\mathbf{b}:\mathbb{Z}_p \rightarrow \mathbb{Z}_p$ is a uniformly distributed random mapping from $\mathbb{Z}_p$ to $\mathbb{Z}_p$.
	  
	$$
	\prob{\left\{ \mathbf{b} \in \mathcal{M}_p:~\A_p(\mathbf{b}) \mbox{ synchronizes} \right\}}
	=
	1- \frac{p!}{p^p}
	=
	1-\Theta\left(\frac{\sqrt{p}}{e^{p}}\right).
	$$
	Thus, a uniformly distributed random circular automaton of prime order $p$ with $k\geq 2$ letters synchronizes with high probability (w.h.p.).
\end{thm}
Theorem~\ref{thm:Perrin} is not explicitly stated in \cite{perrin1977codes}, but it is observed in \cite{MR520853} that 
	Perrin's work implies the theorem.\\

	It is known that the \Cerny conjecture holds true for the class of circular automata \cite{dubuc1998}. In a closely related work, B{\'e}al, Berlinkov and Perrin \cite{beal2011} gave an $O\left(n^2 \right)$ upper bound for the shortest words of synchronizing automata with a single cluster.

 A natural question arises:  do random circular automata of order $n$ (not necessarily prime) synchronize with high probability? We give a positive answer to this question in the following:
\begin{thm}[\textbf{Main result}]\label{thm:main}
	The following holds:
	$$
	\prob{  \left\{  \mathbf{b} \in \mathcal{M}_n:~\mathcal{A}_n (\mathbf{b})\mbox{ synchronizes} \right\} } =
	1 - O\left(\frac{1}{n}\right)
	$$
	as $n \to \infty$. Thus, a randomly chosen $\mathcal{A}_n (\mathbf{b})$ synchronizes w.h.p. as $n \to \infty$.
\end{thm}
\begin{remark}
	Theorem~\ref{thm:main} does not follow from the results of Berlinkov or Nicaud. In their models, they use a random automaton $\A(Q,L)$ of order $n$ where $L$ is a collection of $k$ mappings from $Q$ to $Q$ i.i.d.\ uniformly chosen. For a fixed $k$, the probability of randomly chosen $k$ mappings to contain a permutation with exactly one cycle is bounded from above by $k\cdot \frac{n!}{n^n} \xrightarrow{n \rightarrow \infty} 0$.\\
\end{remark}
%	 We use the following notation throughout the paper, for an integer $r$, we put:
%		$$(r)_n := r \; (\mathrm{mod}\ n),$$ where we always take $(r)_n \in \{0,...,n-1\}.$
Given $n \in \mathbb{N}$ and $r \in \mathbb{Z}$, we define the \textit{$n$-cyclic absolute value} of $r$ to be
$$
	\dist{r}:= \min\left\{ (r)_n,(-r)_n)\right\} \in \left\{0,1,\ldots,\nmed \right\}.
$$
When $r,s \in \mathbb{Z}$ then $\dist{r-s}$ 
%Let $r,s$ be integers and let $n$ be a natural number. We say that
%		$$
%		\dist{r-s}:= \min\left\{ (r-s)_n,(s-r)_n)\right\}
%		$$
		is the \textit{$n$-cyclic distance} between $r$ and $s$.
%		Observe that $\dist{r-s} \in \{0,1,\ldots,\nmed\}$.
		When the numbers $0,1,\ldots,n-1$ are identified with the vertices of a cycle of length $n$, the $n$-cyclic distance between two such numbers is the length of the shortest path between them in the cycle.
		We now introduce the main tool for the proof of the main theorem.
 \begin{defi}
  Let $\mathcal{A}_n (\mathbf{b}) := (\zn,\{\mathbf{a},\mathbf{b}\})$ be a circular automaton with $\mathbf{b} = (b_0,b_1,\ldots,b_{n-1})$.
  %with $n$ states, where the letter $a(i) = (i+1)_n$, and let $ \mathbf{b} = (b_0,b_1,\ldots,b_{n-1})$ be a random vector, where $b_0,b_1,\ldots,b_{n-1}$ are i.i.d.\ uniformly distributed random variables over $\zn$.
%   We associate with $\mathcal{A}_n (\mathbf{b})$ the following matrix of $\nmed$ rows and $n$ columns:
Then we define $T_{\b{}}$ to be the matrix
		\begin{equation}\label{eq:KeyMatrix}
		T_{\b{}}
		:=
		\begin{bmatrix}
			\dist{b_{0}-b_{1}}&\dist{b_{1}-b_{2}}&\ldots&\dist{b_{k}-b_{k+1}}&\ldots&\dist{b_{n-1}-b_{0}}\\
			\dist{b_{0}-b_{2}}&\dist{b_{1}-b_{3}}&\ldots&\dist{b_{k}-b_{(k+2)_n}}&\ldots&\dist{b_{n-1}-b_{1}}\\
			\vdots&\vdots&\ddots&\vdots&\ddots&\vdots\\
			\dist{b_{0}-b_{i}}&\dist{b_{1}-b_{1+i}}&\ldots&\dist{b_{k}-b_{(k+i)_n}}&\ldots&\dist{b_{n-1}-b_{i-1}}\\
			\vdots&\vdots&\ddots&\vdots&\ddots&\vdots\\
			\dist{b_{0}-b_{\nmed}}&\dist{b_{1}-b_{1+\nmed}}&\ldots&\dist{b_{k}-b_{(k+\nmed)_n}}&\ldots&\dist{b_{n-1}-b_{\nmed-1}}
		\end{bmatrix},
		\end{equation}
		shortly written as
		$$ T_{\b{}}(i,j) =  \db{j}{i} \mbox{ for } 1 \leq i \leq \nmed \mbox{ and } 0 \leq j \leq n-1. $$
 \end{defi}
As before, $b_i = \b{}(i)$, i.e., the image of state $i$ under $\mathbf{b}$.
 To be clear, note that the first row of $T_{\b{}}$ is formed of the cyclic distances of the images of states $r,s$ such that $\dist{r-s} = 1$; in general, the $i$-th row of $T_{\b{}}$ is formed of the cyclic distances of the images of pairs of states $r,s$ of cyclic distance $i$. Notice that the columns are counted from $0$ to $n-1$.\\
 
 For $\b{} \in \mathcal{M}_n$ and $i = 1,\ldots,\nmed$, let $R_i (\b{})$ denote the number of different entries in row $i$ of $T_{\b{}}$:
 \begin{equation}\label{eq:RiB}
 R_i (\b{})
 :=
 \#
 \left\{\dist{b_0-b_{(0+i)_n}},
 \dist{b_1-b_{(1+i)_n}},
 \ldots,
 \dist{b_{n-1}-b_{i-1}}\right\}.
 \end{equation} 
Set
 \begin{equation}\label{eq:E_row}
% \mathcal{E}_{\mathrm{\small row}} (\mathbf{b}, \alpha) :=
\RU := \bigcap_{i=1}^{\nmed}\left\{\b{} \in \mathcal{M}_n:~R_i(\b{}) \geq \alpha \nmed \right\},
 \end{equation}
 i.e., $\RU$ contains those $\b{}$ for which every row of $T_{\b{}}$ has at least $\alpha \nmed$ different elements. Its complement is 
\begin{equation}\label{eq:Ec_row}
\mathcal{E}_{\mathrm{\small row}}^c (\alpha) := \bigcup_{i=1}^{\nmed}\left\{\b{} \in \mathcal{M}_n:~R_i(\b{}) < \alpha \nmed \right\}.
\end{equation}
We also define
\begin{equation}\label{eq:E_zero}
\mathcal{E}_{\mathrm{\small zero}} (\beta) := 
\left\{\b{} \in \mathcal{M}_n:~\D(\b{})  \geq  \beta \nmed \right\},
\end{equation}
and its complement  
\begin{equation}\label{eq:Ec_zero}
  \mathcal{E}_{\mathrm{\small zero}}^c (\beta) := 
  \left\{\b{} \in \mathcal{M}_n:~\D(\b{})  < \beta \nmed \right\},
\end{equation}
  where
  $$
  D_i(\mathbf{b})
  :=
  \begin{cases}
  1, &\mbox{ if there exist } \, k,l \in \zn
  \mbox{ such that } \dist{k-l} = i \mbox{ and } \dist{b_k-b_l} = 0;
  \\
  0, &\mbox{ otherwise, }
  \end{cases}
  $$
  and
  $$
  \D(\b{})
  :=
  \sum_{i=1}^{\nmed}
  D_i(\b{}).
  $$
  That is, $\mathcal{E}_{\mathrm{\small zero}} (\beta)$ is the set of those $\b{}$ for which the matrix $T_{\b{}}$ has at least $\beta \nmed$ rows containing the entry zero.\\
  
  The proof of Theorem~\ref{thm:main} relies on the following two lemmas. 
  \begin{lemma}\label{lemma:LowerBoundRowElements}
		Let $\varepsilon>0$ and let $\alpha = 1-e^{-1}-\varepsilon$.
%		$$
%			R_i (\b{})
%			:=
%			\#
%			\left\{\dist{b_0-b_{(0+i)_n}},
%			\dist{b_1-b_{(1+i)_n}},
%			\ldots,
%			\dist{b_s-b_{(s+i)_n}},
%			\ldots,
%			\dist{b_{n-1}-b_{i-1}}\right\}.
%		$$
%		for $i \in \left\{1,\ldots,\nmed\right\}$.
% 		Let $$\left(\mathcal{E}^{1-e^{-1}-\varepsilon}_{row}\right) ^c:= \bigcup_{i=1}^{\nmed}\left\{R_i(\b{}) < \nmed (1-e^{-1} - \varepsilon)\right\}.$$
 		Then
		$$
%		\prob{\left(\mathcal{E}^{1-e^{-1}-\varepsilon}_{row}\right)^c} \\
		\prob{\mathcal{E}_{\mathrm{\small row}}^c (\alpha)}	=
			O\left(\frac{1}{n}\right)
		$$
		as $n \to \infty$.
	\end{lemma}

	\begin{lemma}\label{lemma:ConcentrationOfD}
		Let $\varepsilon \in (0,1)$ and let $\beta = \frac{1}{2}-\varepsilon$.
%		Let
%		$$
%		\mathbb{D}_i(\mathbf{b}) :=
%		\begin{cases}
%			 1, &\mbox{ if there exist } \, k,l \in \zn
%			\mbox{ such that } \dist{k-l} = i \mbox{ and } \dist{b_k-b_l} = 0;
%			\\
%			 0, &\mbox{ otherwise, }
%		\end{cases}
%		$$
%		and let
%		$$
%			\mathbb{D}(\b{})
%			:=
%			\sum_{i=1}^{\nmed}
%			\d{i}(\b{}).
%		$$
%		Finally, let $\left(\mathcal{E}^{0.49(1-\varepsilon)}_{\mbox{\tiny zero}}\right)^c := \left\{\D(\b{})	< 0.49 \nmed (1-\varepsilon) \right\}$.
		Then
		$$
%			\prob{\left(\mathcal{E}^{0.49(1-\varepsilon)}_{\mbox{\tiny zero}}\right)^c}
			\prob{\mathcal{E}_{\mathrm{\small zero}}^c (\beta)}
			=	O\left(\frac{1}{n}\right).
		$$
		as $n \to \infty$.
	\end{lemma}
	
\begin{proof}[\textbf{Proof of Theorem~\ref{thm:main}}]
 The main idea of the proof is to transform the question of synchronization of $\auto(\b{})$ into a question concerning properties of the matrix $T_\mathbf{b}$. The %entries
 functions $T_\mathbf{b}(i,j)$
 %of the matrix
 are random variables over $\mathcal{M}_n$, and to obtain our desired probability estimates we will need to understand the joint stochastic dependence structure of these random variables.

 Let $\mathbf{b} \in \mathcal{M}_n$ and consider the associated Matrix $T_\mathbf{b}$. The first observation is that a zero in row $i$ of $T_\mathbf{b}$ means that two states $r,s$ with cyclic distance $i$ synchronize under $\mathbf{b}$ (i.e., ${\b{}}(r) = {\b{}}(s)$), which implies that any pair $r',s'$ with cyclic distance $i$ can be synchronized with a word of the form $\mathbf{a}^l  \mathbf{b}$ because $\{r',s'\} \mathbf{a}^l = \{r,s\} $ for some $l$.
% = l_{r,s,r',s'}$. 
 The second observation is that if the $i$-th row of $T_\mathbf{b}$ contains a number $j =|b_k - b_{(k+i)_n}|_n$ and the $j$-th row contains a zero, then every pair of states $(r,s)$ with cyclic distance $i$ can be synchronized with a word of the form $\mathbf{a}^{l_1} \mathbf{b} \mathbf{a}^{l_2} \mathbf{b}$. Indeed, we can proceed as follows: $\{r,s\} \stackrel{\mathbf{a}^{l_1}}{\rightarrow}\{k,(k+i)_n\} \stackrel{ \mathbf{b} }{\rightarrow} \{b_k,b_{(k+i)_n}\}$, where this last pair has n-cyclic distance $j$; then $\{b_k,b_{(k+i)_n}\}$ synchronizes with a word of the form $\mathbf{a}^{l_2} \mathbf{b}$, for some $l_2$ because the $j$-th row contains a zero. %thus we can synchronize $\{r,s\}$ with a word of the form $\mathbf{a}^{l_1} \mathbf{b} \mathbf{a}^{l_2} \mathbf{b}$.
 With these two observations, we establish sufficient conditions on $T_\mathbf{b}$ for the synchronization of $\auto(\mathbf{b})$. The sets $\mathcal{E}_{\mathrm{\small row}} (\alpha)$ and $\mathcal{E}_{\mathrm{\small zero}} (\beta)$ which we defined in \eqref{eq:E_row} and \eqref{eq:E_zero} play a crucial role.\\
 Let $\mathbf{b} \in \mathcal{M}_n$. If $\mathbf{b}$ is contained in both $\mathcal{E}_{\mathrm{\small row}} (\alpha)$ and $\mathcal{E}_{\mathrm{\small zero}} (\beta)$ for some $\alpha , \beta>0$ such that $\alpha + \beta> 1$, then $\mathcal{A}_n (\mathbf{b})$ synchronizes. This follows from the two previous observations together with the union bound. Indeed, let $(r,s)$ be any pair of different states and let $i=|r - s|_n$. If row $i$ contains a zero, we can synchronize $\{r,s\}$ with a word of the form $\mathbf{a}^l \mathbf{b}$; otherwise, row $i$ contains an entry $j\neq 0$ such that row $j$ contains a zero  (because $\alpha + \beta > 1$), which implies that $\{r,s\}$ can be synchronized with a word of the form $\mathbf{a}^{l_1} \mathbf{b} \mathbf{a}^{l_2} \mathbf{b}$. Therefore, every pair of different states synchronizes and $\mathcal{A}_n (\mathbf{b})$ synchronizes by Claim~\ref{claim:SynchroEquivalence}. Therefore, for any $\alpha, \beta > 0$ satisfying $\alpha + \beta > 1$, we have the following bound:
\begin{equation}\label{eq:SynchroMatrixReduction}
	\begin{split}
	\prob{ \left\{ \b{} \in \mathcal{M}_n:~\auto(\b{}) \mbox{ synchronizes} \right\} }
	\geq
%	\prob{\RU \wedge \RD}
	\prob{\RU \cap \RD}
	&=
%	1- \prob{\Big( \RU \Big)^c \vee \left(\RD\right)^c }
	1- \prob{\cRU \cup \cRD }
	\\&\geq
	1- \prob{\cRU} - \prob{\cRD}.
	\end{split}
\end{equation}
%		where we use the union bound for the last inequality.
		Now, by the last inequality and by Lemmas~\ref{lemma:LowerBoundRowElements} and \ref{lemma:ConcentrationOfD} we obtain the bound stated in the main theorem. We can choose, for example, $\varepsilon' = 0.05$, $\alpha^{\star} = 1-e^{-1}-\varepsilon'\approx 0.582$ and $\beta^{\star} = 0.5-\varepsilon' = 0.45$, so that $\alpha^{\star}>0$, $\beta^{\star}>0$ and $\alpha^{\star} + \beta^{\star} > 1$. Then we have
		$$
		\prob{ \left\{ \b{} \in \mathcal{M}_n:~\auto(\b{}) \mbox{ synchronizes} \right\} }
		\geq
%		1- \underbrace{\prob{\Big( \RUstar \Big)^c}}_{=O\left(\frac{1}{n}\right)} - \underbrace{\prob{\left(\RDstar\right)^c}}_{= O\left(\frac{1}{n}\right)}
		1- \underbrace{\prob{\cRUstar}}_{=O\left(\frac{1}{n}\right)} - \underbrace{\prob{\cRDstar}}_{= O\left(\frac{1}{n}\right)}
		=
		1 - O\left(\frac{1}{n}\right)
		$$
		as $n \to \infty$.
\end{proof}

%****************************************************
\section{Independence among the random variables $T_{\b{}}(i,j)$}
%entries of $T_{\b{}}$}
\label{sec:DependenceStructureTb}
For every pair $(i,j)$, $1 \leq i \leq \nmed$ and $0 \leq j \leq n-1$,
%$i$ and $j$ in the range from $0$ to $n-1$,
the function $T_\mathbf{b}(i,j):~\mathcal{M}_n \mapsto \zn$ is a random variable on the space $\mathcal{M}_n$, equipped with the uniform probability measure $\mathbb{P}$ (and with the power set of $\mathcal{M}_n$ as the natural sigma-field).
%the matrix entry
%$T_{\b{}}(i,j)$ assigns an integer value (in the sub-range $\{0,1,\ldots,\nmed\}$) to every $\mathbf{b} \in \mathcal{M}_n$. In other words, for every such $i$ and $j$, the function $T_\mathbf{b}(i,j):~\mathcal{M}_n \mapsto \mathbb{Z}$ is an (integer-valued) random variable on the space $\mathcal{M}_n$, equipped with the uniform probability measure $\mathbb{P}$ (and with the power set of $\mathcal{M}_n$ as the natural sigma-field).
It is crucial for our proof to give a criterion on pairs of indices $(i_1, j_1), \dots, (i_k, j_k)$ which guarantees that the random variables $T_\mathbf{b}(i_1,j_1)$, \dots, $T_\mathbf{b}(i_k,j_k)$ are independent. First, notice that not every subset of
%entries of $T_{\b{}}$
random variables $T_{\b{}}(i,j)$ is independent. For example,
$$T_{\b{}}(1,0) = \dist{b_0-b_1}
,\quad
T_{\b{}}(1,1) = \dist{b_1-b_2}
,\quad
T_{\b{}}(2,0) = \dist{b_0 -b_2}$$
are clearly dependent: if the first two random variables $T_{\b{}}(1,0)$ and $T_{\b{}}(1,1)$ are zero, then $b_0 = b_1 = b_2$, which implies that $\dist{b_0 -b_2} = 0$ and so $T_{\b{}}(2,0)$ necessarily is also zero. This dependence comes from the fact that there is a ``cycle'' of the form $b_0 \to b_1 \to b_2 \to b_0$ generated by the indices of these three random variables. Generally, it will turn out that a set of random variables $T_{\b{}}(i,j)$
%entries of $T_{\b{}}$
is independent if and only if the corresponding indices are ``acyclic''. We formalize this in the following
\begin{defi}
%	We call $\{j,(j+i)_n\}$ the \textit{associated edge} of the matrix entry $T_{\b{}}(i,j).$
	Let $$S = \{(i_1,j_1),(i_2,j_2),\ldots,(i_k,j_k)\}$$ be a multi-set, where $i_l,j_l \in \zn$.
	The \textit{associated (multi-)graph} $G(S)$ is the (multi-)graph with vertex set $\zn$ and edge (multi-)set 	$$\Big\{ \{j_1,(j_1+i_1)_n\}, \, \{j_2, (j_2+i_2)_n\},\ldots,\, \{j_k, (j_k+i_k)_n\}  \Big\}.$$
	 We say that $S$ is acyclic if its associated multi-graph $G(S)$
	is acyclic. 
	We also say that the edge $\{j,j+i\}$ is associated to the random variable $T_{\b{}}(i,j)$.
\end{defi}

The relation between acyclic index sets and independent variables is stated in the following

\begin{prop}\label{prop:IndependenceAcyclic}
	The variables $T_{\b{}}(i_1,j_1),T_{\b{}}(i_2,j_2),\ldots,T_{\b{}}(i_k,j_k)$ are i.i.d.\ $\iff$ the (multi-)set $S = \{(i_1,j_1),(i_2,j_2),\ldots,(i_k,j_k)\}$ is acyclic. Furthermore, if the variables are independent then
	\begin{equation}\label{eq:independenceCondition}
		\prob{\bigcap_{w = 1}^{k}\left\{\mathbf{b} \in \mathcal{M}_n:~ \Tb(i_w,j_w) = s_w \right\}}
		=
		\frac{\prod_{w = 1}^{k} m_{s_w}}{n^{k}},
		\quad
		\forall k\geq 1,
	\end{equation}
	 where $s_1,s_2,\ldots,s_k$ are arbitrary integers and
	$$
%		m_s = \# \{ d \in \zn : \dist{1-d} = s \}
		m_s = \# \{ d \in \zn : \dist{d} = s \}
		=
		\begin{cases}
		&2, \quad
		\mbox{ if }
		0 < s < \frac{n}{2};
		%\mbox{ and }
		%s \in \mathbb{N};
		\\
		&1,  \quad
		\mbox{ if }
		s = 0;
		\\
		&1,  \quad
		\mbox{ if }
		s = \frac{n}{2}
		\mbox{ and }
		\frac{n}{2} \in \mathbb{N};
		\\
		&0,  \quad \mbox{ otherwise}.
		\end{cases}
	$$
\end{prop}

Henceforth in the paper we use the concepts ``acyclic'' and ``independent'' interchangeably when we refer to a multi-set of independent random variable entries of $T_{\b{}}$, resp.\ to random variable entries whose associated multi-graph is acyclic.

\begin{remark}\label{rmk:EqualEdges}
Note that different
%entries
random variables $T_{\b{}}(i,j), \, T_{\b{}}(i',j')$
%pairs $(i,j), (i',j')$
may be associated with the same edge;
since $1 \leq i \leq \nmed$
this only happens when $n$ is even and $i=i'=\frac{n}{2}$ and $j \equiv j' \mod \frac{n}{2}.$
Thus, for $n$ odd, a pair of different random variables $T_{\b{}}(i,j), \, T_{\b{}}(i',j')$ 
%entries of $T_{\b{}}$
is always acyclic/independent.
%Thus, for $n$ odd, two different pairs $(i,j), (i',j')$ are always acyclic/independent.
\end{remark}

\begin{remark}\label{rmk:indep}
For a vector $\mathbf{b} \in \mathcal{M}_n$, we can write its entries $b_0, \dots, b_{n-1}$ as functions of $\mathbf{b}$. In other words, $b_0= b_0(\mathbf{b}), \dots, b_{n-1} = b_{n-1} (\mathbf{b})$ are random variables on $\mathcal{M}_n$, equipped with the uniform measure $\mathbb{P}$. The random variables $b_0, \dots, b_{n-1}$ are independent and identically distributed over this space; this follows immediately from the fact that the uniform measure on $\mathcal{M}_n$ is a product of $n$ one-dimensional uniform measures.
\end{remark}

\begin{proof}[\textbf{Proof of Proposition \ref{prop:IndependenceAcyclic}}]
	First note that any two random variables $T_{\b{}}(i,j) = \dist{b_j-b_{{(j+i)}_n}}$ and $T_{\b{}}(i',j') = \dist{b_{j'}-b_{{(j'+i')}_n}}$ are always identically distributed since $b_0,b_1,\ldots, b_{n-1}$ are i.i.d. (see Remark \ref{rmk:indep}). Note also that for all $s$
	$$
	\prob{\left\{\mathbf{b} \in \mathcal{M}_n:~\dist{b_p-b_{p+q}}=s\right\}} = \frac{n\cdot m_s}{n^2}
	= \frac{m_s}{n},
	$$
	which can seen by an easy counting argument: there are $n$ different possible choices of $b_p$, and then there are $m_s$ independent different choices of $b_{(p+q)_n}$ such that $\dist{b_p-b_{p+q}}=s$. Thus equation \eqref{eq:independenceCondition} is just a rephrasing of the fact that the random variables are independent. Therefore, what we need to prove is that independence holds if and only if the associated (multi-)graph is acyclic.\\
	$\Rightarrow)$ (by contraposition) Let $S=\{(i_1,j_1),\, (i_2,j_2),\ldots,\, (i_k,j_k)\}$ be a (multi-)set which
	%is not acyclic.
	contains a cycle.
	Thus, its associated multi-graph $G(S)$ has a cycle $C$ of length $l\geq 2$. Let this cycle be w.l.o.g.\  $$ j_1 \to (j_1+i_1)_n = j_2 \to
	\, (j_2+i_2)_n = j_3 \to \ldots \to (j_{l-1}+i_{l-1}) = j_l \to (j_l + i_l)_n = j_1.$$
	Recall that $T_{\b{}}(i,j) = 0 \iff b_{j} = b_{(j+i)_n}$. Thus if for some $\mathbf{b} \in \mathcal{M}_n$ we have $$T_{\b{}}(i_1,j_1) = T_{\b{}}(i_2,j_2)=\ldots= T_{\b{}}(i_{l-1},j_{l-1}) = 0,$$
	then $b_{j_1} = b_{j_2} = \ldots = b_{j_l}$, and so we automatically also have $T_{\b{}}(i_l,j_l) = \dist{b_{j_l}-b_{(j_l + i_l)_n}}=\dist{b_{j_l}-b_{j_1}}= 0$. Thus, the variables $T_{\b{}}(i_1,j_1), \dots, T_{\b{}}(i_\ell,j_\ell)$ are not independent. We conclude that an independent multi-set must be acyclic.
 	\\  $\Leftarrow)$ (by induction on $k$) Let $k \geq 2$. Assume that the multi-set $S_k=\{(i_1,j_1),\, (i_2,j_2),\ldots,\, (i_k,j_{k})\}$ is acyclic. We want to show that $T_\mathbf{b}(i_{k},j_{k})$ is independent of $T_\mathbf{b}(i_1, j_1), \dots, T_\mathbf{b}(i_{k-1},j_{k-1})$. This will allow us to factor out the $k$-th factor on the left-hand side of \eqref{eq:independenceCondition}, leading (by induction) to the formula on the right-hand side of \eqref{eq:independenceCondition}, which is equivalent to independence.\\  
We distinguish two cases: The first case is when the edge $\{j_k,(j_k+i_k)_n\}$ is a connected component by itself in $G(S)$. This means that the sets $S_1 := \{j_1, (j_1+i_1)_n,j_2, (j_2+i_2)_n, \dots, j_k,(j_{k-1}+i_{k-1})_n\}$ and $S_2:= \{j_{k},(j_{k} + i_{k})_n \}$ are disjoint. By construction, the random variables $T_\mathbf{b}(i_1, j_1), \dots, T_\mathbf{b}(i_{k-1},j_{k-1})$ depend only on $b_s$ with $s \in S_1$, while $T_\mathbf{b}(i_{k},j_{k})$ depends only on $b_s$ with $s \in S_2$. Since $b_0, \dots, b_{n-1}$ are independent by Remark \ref{rmk:indep}, this implies that $T_\mathbf{b}(i_{k},j_{k})$ is independent of $T_\mathbf{b}(i_1, j_1), \dots, T_\mathbf{b}(i_{k-1},j_{k-1})$, as desired.\\
For the second case, the edge $\{j_{k},(j_{k}+i_{k})_n\}$ is not a connected component by itself in $G(S)$. Since it is also not part of a cycle by assumption,we can assume that $(j_k + i_k)_n$ is a leaf vertex in $G(S)$. In principle, $T_\mathbf{b}(i_k,j_k)$ depends on $b_{j_k}$ as well as on $b_{(j_k+i_k)_n}$. However, since $T_\mathbf{b}(i_k,j_k)$ is defined as a cyclic distance, the conditional distribution of $T_\mathbf{b}(i_k,j_k)$ given $b_{j_k}$ is always the same. In formulas, for every $s_k$ we have 
\begin{eqnarray} \label{indep_bk}
 \prob{\left\{\mathbf{b} \in \mathcal{M}_n:~T_\mathbf{b} (i_k,j_k) = s_k \right\} }  =  \prob{\left\{\mathbf{b} \in \mathcal{M}_n:~T_\mathbf{b} (i_k,j_k) = s_k \text{ and } b_{j_k} = r \right\} }
\end{eqnarray}
for every $r \in \{0, \dots, n-1\}$. This fact can be simply established by counting the possible configurations of $b_{j_k}$ and $b_{(j_k+i_k)_n}$. By definition, $T_\mathbf{b}(i_k,j_k)$ is independent of all $b_\ell$ with $\ell \neq j_k, (j_k+i_k)_n$. Thus for every numbers $s_1, \dots, s_k$ we have, using the independence of $b_0, \dots, b_{n-1}$ and \eqref{indep_bk}, that
\begin{eqnarray*}
& & \prob{\bigcap_{w=1}^k \left\{\mathbf{b}:~T_\mathbf{b} (i_w,j_w) = s_w \right\} } \\
& = & \sum_{r=0}^{n-1} \prob{\bigcap_{w=1}^k \left\{\mathbf{b}:~T_\mathbf{b} (i_w,j_w) = s_w \text{~and~}b_{j_k}=r \right\} } \\
& = & \sum_{r=0}^{n-1} \prob{ \underbrace{\left(\bigcap_{w=1}^{k-1} \left\{\mathbf{b} :~T_\mathbf{b} (i_w,j_w) = s_w \text{~and~}b_{j_k}=r \right\} \right)}_{\text{depends only on $b_\ell$ with $\ell \neq j_k, (j_k+i_k)_n$ when $b_{j_k}$ is fixed}} \cap \underbrace{\left\{ \mathbf{b}:~T_\mathbf{b} (i_k,j_k) = s_k \text{ and } b_{j_k} = r \right\}}_{\text{depends only on $b_{(j_k+i_k)_n}$ when $b_{j_k}$ is fixed}} } \\
& = & \sum_{r=0}^{n-1} \left( \prob{ \left(\bigcap_{w=1}^{k-1} \left\{\mathbf{b} :~T_\mathbf{b} (i_w,j_w) = s_w \text{~and~}b_{j_k}=r \right\} \right)} \prob{ \left\{\mathbf{b} :~ T_\mathbf{b} (i_k,j_k) = s_k \text{ and } b_{j_k} = r \right\} } \right) \\
& = & \sum_{r=0}^{n-1} \left( \prob{ \left(\bigcap_{w=1}^{k-1} \left\{\mathbf{b} :~T_\mathbf{b} (i_w,j_w) = s_w \text{~and~}b_{j_k}=r \right\} \right)} \prob{ \left\{ \mathbf{b} :~ T_\mathbf{b} (i_k,j_k) = s_k \right\} } \right) \\
& = & \prob{ \left(\bigcap_{w=1}^{k-1} \left\{\mathbf{b} :~T_\mathbf{b} (i_w,j_w) = s_w \right\} \right)} \prob{ \left\{ \mathbf{b} :~ T_\mathbf{b} (i_k,j_k) = s_k\right\} } .
\end{eqnarray*}
This is exactly the independence property that we wanted to establish.
\end{proof}

%****************************************************
\section{Proof of Lemma~\ref{lemma:LowerBoundRowElements}}\label{sec:ProofLemmaRows}
 The overview of the proof is as follows. Recall that we understand the entries of the matrix $T_\mathbf{b}$ as random variables. We will prove that every row of $\Tb$ contains a ``large'' number of independent random variables. Then we give a lower bound for the expected value of the number of different elements in each row. Then we apply McDiarmid's inequality to each row and finally we use the union bound together with the exponential decay delivered by McDiarmid's inequality to guarantee that w.h.p. every row of $\Tb$ has at least $\sim (1-e^{-1}) \nmed$ different elements.
We denote by $C_n(i)$ the \textit{circulant graph} on $n$ vertices, i.e., the graph with vertex set $\zn$ where two vertices $r,s$ are adjacent if and only if $\dist{r-s} = i.$

We need the following property.
\begin{claim}\label{claim:SubsetIndependent}
	For every $i$, the $i$-th row of $\Tb$ contains a set of at least $n-\gcd(n,i)$ random variables which are i.i.d.
\end{claim}
\begin{proof}[\textbf{Proof}] The variables in row $i$ are given by the multi-set
\begin{equation} \label{E_i_def}
E_i(\b{}) := \{\dist{b_0-b_{i}},
\ldots,
\dist{b_k-b_{(k+i)_n}},
\ldots,
\dist{b_{n-1}-b_{i-1}}
\}.
\end{equation}
Let $i \neq \med$. By Remark~\ref{rmk:EqualEdges},
the corresponding multi-set $E_i(\b{})$
%the associated multi-graph $G(S_i)$ of the set of indices of the $i$-th row of $\Tb$ 
%$ S_i = \{(i,0),(i,1),\ldots,(i,n-1)\}$$
does not have repeated elements and
%edges and
the associated multi-graph $G(E_i(\b{}))$
is isomorphic to the circulant graph $C_{n}(i)$. It is well known and easy to show that  $C_{n}(i)$ is a disjoint union of $\gcd(n,i)$ cycles of length $\frac{n}{\gcd(n,i)}$ \cite{BoeschConnectivitiesCirculants}. We can then obtain an acyclic set of variables by removing one edge from each of the cycles of $G(S_i)$.
%variables that correspond to exactly one edge of each cycle in $G(S_i)$.
%$G(E_{i}(\b{}))$.
The resulting set of variables is i.i.d.\ by Proposition~\ref{prop:IndependenceAcyclic}.  In the case $i =\med$, the first $\med$ variables in row $\med$
$$
E_{\med}(\b{}) = \{\dist{b_0-b_{\med}},
\ldots,
\dist{b_k-b_{(k+\med)_n}},
\ldots,
\dist{b_{\med-1}-b_{n-1}}
\}
$$
%and the associated multi-graph of the corresponding set of indices 
have an associated multi-graph that
is isomorphic to the circulant graph $C_{n}(\med)$, which is a disjoint union of $\med = \gcd(n,\med)$ edges. This last graph is acyclic, thus the variables are i.i.d.\ by Proposition~\ref{prop:IndependenceAcyclic}.
\end{proof}
 We prove the following lower bound
\begin{claim}\label{claim:LowerBoundERi}
	We have $\E{R_i}
	\geq
	\nmed (1-e^{-1}) - 1$, where for all $\b{} \in \mathcal{M}_n$ 
	$$R_i(\b{}) = \# \{\dist{b_0-b_{(0+i)_n}},
	\ldots,
	\dist{b_k-b_{(k+i)_n}},
	\ldots,
	\dist{b_{n-1}-b_{i-1}}
	\}$$
	(see \eqref{eq:RiB})
	is the cardinality of different elements in row $i$ of $\Tb$.
\end{claim}
\begin{proof}[\textbf{Proof}]
 First, for every $d \in \{0, \dots, \nmed\}$, we define the random variables
$$
	\delta_j^{(i)}(\b{},d)
	:=
	1-\mathds{1}\{\dist{b_j-b_{(j+i)_n}}= d\}
	=
	\begin{cases}
		&0, \quad \mbox{ if } \dist{b_j-b_{(j+i)_n}}= d;
		\\&1, \quad \mbox{ otherwise.}
	\end{cases}
$$
and
$$	
	r_{d}^{(i)} (\b{})
	:=
	\prod_{j\in \zn}
		\delta_j^{(i)}(\b{},d)
	=
	\begin{cases}
%		& \quad  \mbox{ if } \exists \, p,q
%		\in \zn \mbox{ such }
%		\\
%		&0, \mbox{ that } d_n(p,q) = i
%		\\
%		& \quad \mbox{ and }
%		d(\b{p}, \b{q}) = d;
		&0 \quad  \mbox{ if } \exists \, p,q \in \zn \mbox{ such that } 
		%d_n(p,q) = i \mbox{ and } d(\b{p}, \b{q}) = d;
		\dist{p,q} = i \mbox{ and } \dist{b_{p}, b_{q}} = d;
		\\
		&1, \quad \mbox{ - otherwise}.
	\end{cases}
$$
Note that $r_{d}^{(i)} (\b{})$ is zero if the number $d$ is included in the $i$-th row of $\Tb$, and that it is one otherwise. Recalling that the entries of $T_{\b{}}$ can only have values in $\{0,1,\ldots,\nmed\}$, we write the number of distinct elements in row $i$ as
\begin{equation}\label{eq:RiAlternateForm}
	R_i(\b{})
	=
	\left(\nmed + 1\right)
	-
	\sum_{d = 0}^{\nmed}
		r_d^{(i)} (\b{}).
\end{equation}
By Claim~\ref{claim:SubsetIndependent}, there is a subset $I$ of $\zn$ of cardinality $n-\gcd(n,i)$ such that the variables $\{ \delta_w^{(i)} : w\in I \}$ are i.i.d., and thus
$$
	\E{r_d^{(i)}}
	=
	\E{\prod_{j\in\zn} \delta_j^{(i)}(\b{},d)}
	\leq
	\E{\prod_{w\in I}\delta_w^{(i)}(\b{},d)}
	= \E{\delta_0^{(i)}(\b{},d)}^{n-\gcd(n,i)}.
$$
Furthermore, by Proposition~\ref{prop:IndependenceAcyclic}, we have
$
	\E{\delta_0^{(i)}(\b{},d)}
	=
	1 - \frac{m_d}{n},
$
and thus
$$
	\E{r_d^{(i)}}
	\leq
	\left(
		1 - \frac{m_d}{n}
	\right)^{n-\gcd(n,i)}
	\leq
	\left(
	1 - \frac{m_d}{n}
	\right)^{\frac{n}{2}}
	=
	\begin{cases}
		&
		\left(1 -\frac{2}{n}\right)^{\frac{n}{2}}
		,\mbox{ if } d\neq 0, \frac{n}{2};
		\\
		&
		\left(1 -\frac{1}{n}\right)^{\frac{n}{2}},
		\mbox{ otherwise },
		\\
	\end{cases}
$$
for $d\in \{0,1,\ldots, \nmed \}$. Using the inequality $1-x\leq e^{-x}$, which is valid for any real number $x$, we obtain
$$
\E{\sum_{d = 0}^{\lfloor \frac{n}{2} \rfloor}
	r_d^{(i)}}
\leq
\nmed
\underbrace{\left(1 -\frac{2}{n}\right)^{\frac{n}{2}}}_{\leq e^{-1}}
+
2
\underbrace{\left(1 -\frac{1}{n}\right)^{\frac{n}{2}}}_{\leq e^{-\frac{1}{2}}}
\leq
\nmed e^{-1} + 2.
$$
Plugging this inequality into \eqref{eq:RiAlternateForm} yields
$$
\E{R_i}
=
\left(\nmed
+
1\right)
-
\E{\sum_{d = 0}^{\lfloor \frac{n}{2} \rfloor}
r_d^{(i)}}
\geq
\nmed (1-e^{-1}) - 1.
$$
This proves Claim~\ref{claim:LowerBoundERi}.
\end{proof}
 We introduce McDiarmid's inequality to prove Claim~\ref{claim:ElementsInRow}.
\begin{defi}
	 Let $L: \left(\zn \right)^n \rightarrow \mathbb{R}$ be a function. We say that $L$ has \textit{Lipschitz coefficient}  $r \in \mathbb{R}^{+}$
	 if
	 $$
	  |L(\overrightarrow{v})-L(\overrightarrow{w})|
	  \leq
	  r
	 $$
	 for every $\overrightarrow{v},\overrightarrow{w} \in \left(\zn \right)^n$ such that $\overrightarrow{v}(j) = \overrightarrow{w}(j)$ for all $j$ except for at most one index.
\end{defi}
%\begin{remark*}
%	The definition of Lipschitz coefficient is more general than the one presented here.
%\end{remark*}
\begin{prop}[\textbf{McDiarmid's Inequality}  \cite{mcdiarmid1989method}]\label{prop:McDiarmid}
	Let $\bar{X} := (X_1,X_2,\ldots,X_n) \in \left(\zn \right)^n$ be a random vector, where the variables $X_1,X_2,\ldots,X_n$ are independent, and let $L: \left(\zn \right)^n  \rightarrow \mathbb{R}$ be a function with bounded Lipschitz coefficient $r$. Then	
\begin{flalign*}
		&\text{(lower tail)}
		\qquad \prob{ L(\bar{X}) \leq \E{L(\bar{X})} - r \sqrt{\lambda n}}
		\leq
		e^{-2\lambda},
\end{flalign*}	
	for all $\lambda \geq 0$.
\end{prop}
\begin{remark*}
	This is just a special case of the general form of McDiarmid's inequality. The general inequality also bounds the upper tail, and allows different Lipschitz coefficients in the respective components.
\end{remark*}
 In the following claim we use Proposition~\ref{prop:McDiarmid} to estimate the probability that row $i$ of $\Tb$ has less than $\sim (1-e^{-1}) \nmed$ different elements.
\begin{claim}\label{claim:ElementsInRow} Let $\varepsilon > 0$. Then
	$$
	\prob{\mathbf{b} \in \mathcal{M}_n:~R_i(\b{}) < \nmed (1-e^{-1} - \varepsilon)}
	\leq
	e^{-\Theta(n)},
	$$
	for $i = 1,2,\ldots,\nmed$.
	%where $\Theta(n)$ is independent of $i$.
\end{claim}
\begin{proof}[\textbf{Proof}]
	 Let $\mathbf{b}=(b_0,b_1,\ldots,b_{n-1})$. Let $E_i(\mathbf{b})$ be defined as in \eqref{E_i_def}. The function $R_i(\mathbf{b}) := \# E_i(\mathbf{b})$ has Lipschitz coefficient 2: changing one $b_j$ affects at most two entries, namely $\dist{b_j-b_{(j+i)_n}}$ and $\dist{b_{(j-i)_n}-b_{j}}$. Using McDiarmid's inequality, we deduce that
	$$
	\prob{\mathbf{b} \in \mathcal{M}_n:~R_i(\b{}) \leq \E{R_i} - 2\sqrt{\lambda n}}
	\leq
	e^{-2\lambda},\quad \, \forall \lambda \geq 0.
	$$
	Using the lower bound $\E{R_i} \geq \nmed (1-e^{-1}) -1 $ of Claim~\ref{claim:LowerBoundERi} we obtain
	\begin{eqnarray*}
	& & \prob{\mathbf{b} \in \mathcal{M}_n:~R_i(\b{}) < \left(\nmed (1-e^{-1}) - 1\right) - 2\sqrt{\lambda n}}  \\
	& \leq & \prob{\mathbf{b} \in \mathcal{M}_n:~R_i (\b{}) \leq \E{R_i} - 2\sqrt{\lambda n}} \\
	& \leq &	e^{-2\lambda},\quad \, \forall \lambda \geq 0.
	\end{eqnarray*}
	Let $\varepsilon>0$
	%be arbitrary but fixed,
	and let
	\begin{equation}\label{eq:Lamdavarepsilon}
	\lambda_{\varepsilon}(n)
	:=
	\frac{1}{4n}\left( \varepsilon \nmed - 1 \right)^2
	=
	\Theta(n);
	\end{equation}
	we observe that $\lambda_{\varepsilon}(n)$ is independent of $i.$ Let $n > \frac{2}{\varepsilon}$, then plugging $\lambda = \lambda_{\varepsilon}(n) $ into the previous inequality yields
\begin{equation}
	\prob{\mathbf{b} \in \mathcal{M}_n:~R_i(\b{}) < \nmed (1-e^{-1} - \varepsilon)}
	\leq
	e^{-2\lambda_{\varepsilon}(n)}
	=
	e^{-\Theta(n)}.
\end{equation}
\end{proof}
Recall that $\RU$ contains those $\mathbf{b} \in \mathcal{M}_n$ for which every row of $\Tb$ has at least $\alpha \nmed$ different elements, so that
$$\cRU = \bigcup_{i = 0}^{\nmed} \left\{  \mathbf{b} \in \mathcal{M}_n:~R_i(\b{}) < \alpha \nmed
\right\}.$$
Let $\varepsilon > 0$ be arbitrary and let $\alpha^* = 1-e^{-1}-\varepsilon$. Then
\begin{eqnarray}
\prob{\cRUstar} & = & \prob{\bigcup_{i=1}^{\nmed} \left\{\mathbf{b} \in \mathcal{M}_n:~ R_i(\b{}) < \alpha^\star \nmed \right\}} \nonumber\\
& \leq & 
\sum_{i = 1}^{\nmed} \prob{\mathbf{b} \in \mathcal{M}_n:~ R_i(\b{}) < \alpha^\star \nmed} \nonumber\\
& \leq & 
n  e^{-\Theta(n)}, \label{eq:ExplicitBoundRows}
\end{eqnarray}
where we use 
%the union bound for the first inequality and
Claim~\ref{claim:ElementsInRow} for the second inequality. 
The proof of Lemma~\ref{lemma:LowerBoundRowElements} then follows by noticing that
$$
n e^{-\Theta(n)}=O\left(\frac{1}{n}\right).
$$

%****************************************************
\section{Proof of Lemma~\ref{lemma:ConcentrationOfD}}\label{sec:ProofLemmaZeros}
The overview of the proof is as follows. We will define two random variables $\Z_0(\b{})$ and $\Z_1(\b{})$ such that
\begin{align*}
	&\bullet \D(\mathbf{b}) \geq \Z_0(\mathbf{b})-\Z_1(\mathbf{b}),\quad \forall \, \b{}:\zn \rightarrow \zn;
	\\&\bullet \E{\Z_0-\Z_1} \sim \frac{n}{2}.
\end{align*}
Then we will show that $\Z_0$ and $\Z_1$ concentrate around their respective means, and use this fact to give an upper bound on the probability that $\D$ is small. For this purpose, we note the following property.
\begin{claim}\label{claim:ConcentrationReduction}
Let $\Z_0, \Z_1$ and $\D$ be random variables which take non-negative values, such that $\D \geq \Z_0-\Z_1$.
Let $\nu > 0$ and let $\delta \leq \E{\Z_0 -\Z_1} - 2 \nu$.
%If $\Z_0-\Z_1 \leq \D$, then
Then
$$
\prob{\D < \delta}
\leq
\prob{\Z_0 < \E{\Z_0} - \nu}
+
\prob{\Z_1 > \E{\Z_1} + \nu}.
$$
\end{claim}

\begin{proof}[\textbf{Proof}]
	 This follows easily from the assumption that $\Z_0 - \Z_1 \leq \D$ and the union bound.
\end{proof}

 \noindent To prove concentration of $\Z_0$ and $\Z_1$ around their respective means, we use Chebyschev's inequality. Notice that	$\D:\zn^n \rightarrow \zn$ does not have a bounded Lipschitz coefficient, so we cannot use McDiarmid's inequality to guarantee its concentration.

\subsection{ Lower bound for $\D(b)$ }\label{subsec:LowerBoundD}
Recall that $\D(\mathbf{b})$ counts the number of rows of $T_\mathbf{b}$ that contain at least one zero. Let
$$
	z_i
	=
	z_i(\mathbf{b})
	:=
	\#(\mbox{Zeros in row $i$ of $T_\mathbf{b}$})
$$
and
$$
\Z_0(\mathbf{b})
:=
\#(\mbox{Zeros in } T_\mathbf{b})
=
\sum_{(i,j) \in [1,\nmed] \times [0,n-1]}
\indi{T_\mathbf{b}(i,j)=0}.
$$
Then
\begin{equation}\label{eq:DMaxForm}
 \D(\mathbf{b})
 =
 \Z_0  (\mathbf{b})
 -
 \sum_{i=1}^{\nmed}
 \max(z_i -1, 0)   .
\end{equation}
It is easy to verify that the number of non-ordered pairs of entries in the $i$-th row with zero value is
$$
	\sum_{0 \leq j < j' \leq n-1}
	\indi{T_\mathbf{b}(i,j) = 0}
	\indi{T_\mathbf{b}(i,j') = 0}
	=
	\frac{z_i(z_i - 1)}{2}
	\geq
	\max(z_i-1, 0),
	\quad
	\forall i,
$$
therefore
$$
	\Z_1(\mathbf{b})
	:=
	\sum_{i = 1}^{\nmed}
	\sum_{0 \leq j < j' \leq n-1}
	\indi{T_\mathbf{b}(i,j) = 0}
	\indi{T_\mathbf{b}(i,j') = 0}
	\geq
	\sum_{i=1}^{\nmed}
	\max(z_i -1, 0).
$$
From this and \eqref{eq:DMaxForm}, we conclude that

\begin{claim} \label{claim:LowerBoundD}
$
	\D(\mathbf{b})
	\geq
	\Z_0(\mathbf{b}) - \Z_1(\b{}),
	\quad                               
	\forall \, \b{}:\zn \rightarrow \zn.
$
\end{claim}

\subsection{Estimates for $\E{\Z_0}$, $\E{\Z_1}$,  $\E{\Z_0 - \Z_1}$,  $\var{\Z_0}$, $\var{\Z_1}$}
In this subsection we prove that
\begin{itemize}
\item $\E{\Z_0-\Z_1}\sim \frac{n}{2},$
\item $\E{\Z_0} = \Theta(n),$
\item $\E{\Z_1} = \Theta(n),$
\item $\var{\Z_0} = O(n),$ and
\item $\var{\Z_1} = O(n)$.
\end{itemize}

\noindent For the rest of this subsection, we use the notation
$$
	y_{i,j}
	=
	y_{i,j}(\b{})
	:=
	\indi{\Tb(i,j) = 0},
$$
for $1\leq i\leq \nmed$ and $0 \leq j \leq  n-1.$
\begin{defi}
	The variables $y_{i_1,j_1},y_{i_2,j_2}\ldots, y_{i_k,j_k}$ are called \textit{acyclic} if the multi-set $ \bigcup_{w=1}^{k}\{(i_w,j_w)\}$ is acyclic. Let $$
	G\left(\{y_{i_1,j_1},y_{i_2,j_2}\ldots, y_{i_k,j_k}\}\right) = G\left(\bigcup_{w=1}^{k}\{(i_w,j_w)\}\right)$$ be the \textit{associated multi-graph} of the multi-set $\{y_{i_1,j_1},y_{i_2,j_2}\ldots, y_{i_k,j_k}\}$ and let $e(y_{i,j}):= \{j,(j+i)_n\}$ be the \textit{associated edge} to $y_{i,j}$. The \textit{length} of $e(y_{i,j})$ is $ \dist{j-(j+i)_n} = i.$
\end{defi}
\begin{remark}\label{rmk:IndepenceYij}
	If the variables  $y_{i_1,j_1},y_{i_2,j_2}\ldots, y_{i_k,j_k}$ are acyclic then they are i.i.d.; this is an immediate consequence of Proposition~\ref{prop:IndependenceAcyclic}.
\end{remark}
 We begin with the easy part: the bounds for the expected values.
\begin{claim}
	\label{claim:FirstMoments}
	Let $n \in \mathbb{N}$. We have $\E{\Z_0} = \Theta(n)$, $\E{\Z_1} = \Theta(n)$, and
 	$\E{\Z_0-\Z_1} \geq \frac{1}{2} \nmed - 1.$
\end{claim}

\begin{proof}[\textbf{Proof}]
	 Using the linearity of the expectation, we get that
	
\begin{equation}\label{eq:EZ0}
	\E{\Z_0}
	=
	\sum_{(i,j) \in [1,\nmed] \times [0,n-1]}
	\E{y_{i,j}}
	=
	\nmed n \frac{1}{n}
	=
	\nmed
	=
	\Theta(n),
\end{equation}
	where for the second equality we use that
	\begin{equation}\label{eq:Ey}
	\E{y_{i,j}}
	=
	\prob{\{\mathbf{b}:~T_\mathbf{b} (i,j) = 0\} }
	=
	\prob{\{ \mathbf{b}:~{b_j = b_{(j+i)}} \}}
	=
	\frac{1}{n}.
	\end{equation}
	Now we calculate an upper bound for $\E{\Z_1}$, depending on the parity of $n$.\\
	
	Case 1: $n$ odd. Every product $y_{i,j}y_{i,j'}$ in the sum
	$$
	\Z_1
	=
	\sum_{i = 1}^{\nmed}
	\sum_{0 \leq j < j' \leq n-1}
	y_{i,j}y_{i,j'}
	$$
	is formed of independent random variables $y_{i,j}$,  $ y_{i,j'}$ by Remarks~\ref{rmk:EqualEdges},\ref{rmk:IndepenceYij}. Thus
	\begin{align*}
	\E{\Z_1}
	=
	\sum_{i = 1}^{\nmed}
	\sum_{0 \leq j < j' \leq n-1}
	\E{y_{i,j}y_{i,j'}}
	&=
	\sum_{i = 1}^{\nmed}
	\sum_{0 \leq j < j' \leq n-1}
	\E{y_{i,j}}\E{y_{i,j'}}
	\\&\stackrel{\mbox{\tiny \eqref{eq:Ey}}}{=}
	\nmed {n \choose 2}
	\frac{1}{n^2}
	\\&=
	\underbrace{\frac{1}{2}\nmed \left(1 - \frac{1}{n}\right)}_{\leq \frac{1}{2}\nmed}
	=
	\Theta(n).
	\end{align*}
	
	Case 2: $n$ even. Using Remark~\ref{rmk:EqualEdges}, we write $\Z_1$ as
	$$
	\Z_1
	=
	\sum_{\substack{1\leq i<\med \\0 \leq j < j' \leq n-1}}
	y_{i,j}y_{i,j'}	
	+
	\sum_{\substack{0\leq r < r' \leq n-1\\ r\not\equiv r' \Mod{n/2}}}
	y_{n/2,r}y_{n/2,r'}
	+
	\sum_{s = 0}^{\frac{n}{2} - 1}
	y_{n/2,s}.
	$$
	Every product $y_{i,j}y_{i,j'}$ in the first sum is formed of independent variables $y_{i,j}$, $y_{i,j'}$ by Remark~\ref{rmk:EqualEdges} and the same is valid for the products $y_{\med,r}y_{\med,r'}$ in the second sum, therefore
	\begin{align*}
	\E{\Z_1}
	&=
	\sum_{i = 1}^{\frac{n}{2} - 1}
	\sum_{0 \leq j < j' \leq n-1}
	\E{y_{i,j}}\E{y_{i,j'}}	
	+
	\sum_{\substack{0\leq r < r' \leq n-1\\ r\not\equiv r' \Mod{n/2}}}
	\E{y_{n/2,r}}\E{y_{n/2,r'}}
	+
	\sum_{s = 0}^{\frac{n}{2} - 1}
	\E{y_{n/2,s}}
	\\&=
	\sum_{i = 1}^{\frac{n}{2} - 1}
	\sum_{0 \leq j < j' \leq n-1}
	\frac{1}{n^2}	
	+
	\sum_{\substack{0\leq r < r' \leq n-1\\ r\not\equiv r' \Mod{n/2}}}
	\frac{1}{n^2}
	+
	\sum_{s = 0}^{\frac{n}{2} - 1}
	\frac{1}{n}
	\\&=
	\left( \frac{n}{2} - 1 \right)
	\cdot
	{n \choose 2}
	\cdot
	\frac{1}{n^2}
	+
	\left( {n \choose 2} - \frac{n}{2} \right)
	\cdot \frac{1}{n^2}
	+
	\frac{n}{2}\cdot\frac{1}{n}
	\\& =
	\underbrace{	\frac{1}{2} \cdot \frac{n}{2} \cdot
	\left(
	1-\frac{1}{n}
	+\left( \frac{2}{n}
	- \frac{2}{n^2}\right)
	\right)}_{\leq
	\frac{1}{2} \cdot \frac{n}{2} +1}
	=
	\Theta(n).
	\end{align*}
	We deduce from the previous cases that $\E{\Z_1} = \Theta(n)$ and $\E{\Z_1} \leq \frac{1}{2} \nmed + 1$ for all $n$. Using this last inequality and \eqref{eq:EZ0}, we conclude that
	\begin{equation*}\label{eq:LowerBoundEZ}
	\E{\Z_0}-\E{\Z_1} = \nmed - \E{\Z_1} \geq \frac{1}{2} \nmed - 1.
	\end{equation*}
	This concludes the proof of Claim~\ref{claim:FirstMoments}.
\end{proof}

 Now we estimate the variance of $\Z_0$ and $\Z_1$.

\begin{claim}\label{claim:SecondMoments} Let $n \in \mathbb{N}$, then
$\var{\Z_0} = O(n)$ and    $\var{\Z_1} = O(n)$.
\end{claim}
\begin{proof}[\textbf{Proof}]
Here we also divide the calculations according to the parity of $n$.\\

 Case 1: $n$ odd. We expand the variance of $\Z_0$ to get that
$$
	\var{\Z_0}
	=
	\sum_{\substack{
			1\leq i \leq \nmed
			\\
			0 \leq j \leq n-1}}
		\var{y_{i,j}}
	+
	\sum_{\substack{
			1\leq i,i' \leq \nmed
			\\
			0 \leq j, j' \leq n-1\\
			(i,j)\neq(i',j')}}
	\cov{y_{i,j},y_{i',j'}},
$$
where the covariances are calculated among pairs of independent variables $y_{i,j},y_{i',j'}$ due to Remark~\ref{rmk:EqualEdges}. Thus
$$
\var{\Z_0}
=
\sum_{\substack{
		1\leq i \leq \nmed
		\\
		0 \leq j \leq n-1}}
\var{y_{i,j}}.
$$
We notice that $y_{i,j}^2 = y_{i,j}$ because $y_{i,j} \in \{0,1\}$, therefore
\begin{equation}\label{eq:Vy}
 \var{y_{i,j}}
 =
 \E{y_{i,j}^2} - \E{y_{i,j}}^2
 =
 \frac{1}{n}
 -
 \frac{1}{n^2}, \quad \forall \, n \in \mathbb{N},
\end{equation}
where we use \eqref{eq:Ey} in the last equality. Then, for all $n$ odd, we get that
\begin{equation}\label{eq:VZ0OddO(n)}
\var{\Z_0}
=
\nmed
n
\left(	 \frac{1}{n}
	-
	\frac{1}{n^2}\right)
=
\nmed \left( 1- \frac{1}{n} \right)
= O(n).
\end{equation}
Now we calculate
\begin{equation}\label{eq:VZ1OddDef}
\var{\Z_1}
=
\sum_{\substack{1 \leq i\leq \nmed \\0 \leq j < j' \leq n-1}}
\var{y_{i,j}y_{i,j'}}
+
\sum_{\sumOne}
\cov{y_{i,j}y_{i,j'},y_{r,s}y_{r,s'}};
\end{equation}
We first note that
\begin{equation}
	\var{y_{i,j}y_{i,j'}}
	=
	\E{y_{i,j}^2 y_{i,j'}^2}
	-
	\E{y_{i,j} y_{i,j'}}^2
	=
	\frac{1}{n^2}
	-
	\frac{1}{n^4},
	\quad
	\mbox{ for } n
	\mbox{ odd and }
	\forall i \mbox{ and }j\neq j';
\end{equation}
this follows since the variables $y_{i,j}$ and $y_{i,j'}$ are different and therefore independent (see Remark~\ref{rmk:EqualEdges}). Thus
\begin{equation}\label{eq:UpperBoundFirstSumVZ1Odd}
\sum_{\substack{1 \leq i\leq \nmed \\0 \leq j < j' \leq n-1}}
	\var{y_{i,j}y_{i,j'}}
=
\nmed
{n \choose 2}
\frac{1}{n^2}
\left(
	1-\frac{1}{n^2}
\right)
=
O(n).
\end{equation}

 For the sum of the covariances, we proceed as follows: if the variables $y_{i,j},y_{i,j'},y_{r,s},y_{r,s'}$ are acyclic then they are independent (see Proposition~\ref{prop:IndependenceAcyclic}), therefore
$$
\cov{y_{i,j}y_{i,j'},y_{r,s}y_{r,s'}} = 0.
$$
On the other hand, if the variables $y_{i,j},y_{i,j'},y_{r,s},y_{r,s'}$ are not acyclic, let
$$
\mathcal{Y} :=
\left\{
	\{y_{i,j},y_{i,j'},y_{r,s},y_{r,s'}\} : (i,j,j') \neq (r,s,s'), \, j < j',\,
	s< s'
\right\},
$$ 
and let  $$Y = \{y_{i,j},y_{i,j'},y_{r,s},y_{r,s'}\} \in \mathcal{Y}.$$ Then  $G(Y)$ is a multi-graph with four edges $e(y_{i,j}),e(y_{i,j'}),e(y_{r,s}),e(y_{r,s'})$ such that $e(y_{i,j})\neq e(y_{i,j'})$ and $e(y_{r,s})\neq e(y_{r,s'})$ (see Remark~\ref{rmk:EqualEdges}). In particular, there cannot be 3 equal edges. If $G(Y)$ has at least one cycle, it is isomorphic to one of the multi-graphs in Figure~\ref{fig:Nodd} below.    \\
	\begin{figure}[!htbp]
		\centering
	% Primer dibujo
	\begin{minipage}[b]{\porcion \textwidth}
		\centering
		\begin{tikzpicture}
		\tikzstyle{every node}=[draw,circle,fill=black,minimum size=2.5pt,
			inner sep=0pt]
			% Primer dibujo
			\draw (0,0) node (11) [] {}
			-- ++(0:\largo ) node (12) [] {}
			-- ++(90:\largo ) node (13) [] {}
			-- ++(180:\largo ) node (14)[] {}
			-- (11);
			\end{tikzpicture}
			\caption*{\tiny $G_1$}	
	\end{minipage}
	\hspace{0.5cm}
	%2 dibujo
	\begin{minipage}[b]{\porcion \textwidth}
		\centering
		\begin{tikzpicture}
		\tikzstyle{every node}=[draw,circle,fill=black,minimum size=2.5pt,
		inner sep=0pt]
		\draw (0,0) node (21) [] {}
		-- ++(0:\largo) node (22) [] {}
		-- ++(120:\largo) node (23) [] {}
		-- (21);
		\draw (22)
		-- ++(90:0.9*\largo) node (24)[]{};
		\end{tikzpicture}
		\caption*{\tiny$G_2$}
	\end{minipage}
	\hspace{0.5cm}
	% 3 dibujo
	\begin{minipage}[b]{\porcion \textwidth}
		\centering
		\begin{tikzpicture}
		\tikzstyle{every node}=[draw,circle,fill=black,minimum size=2.5pt,
		inner sep=0pt]
		\draw (0,0) node (1) [] {}
		-- ++(0:\largo) node (2) [] {}
		-- ++(120:\largo) node (3) [] {}
		-- (1);
		\draw (1.2*\largo,0) node (4)[]{}
		-- ++(90:\largo) node (5)[]{};
		\end{tikzpicture}
		\caption*{\tiny$G_3$}
	\end{minipage}
	\hspace{0.5cm}
		% 4 dibujo
	\begin{minipage}[b]{\porcion \textwidth}
		\centering
		\begin{tikzpicture}
		\tikzstyle{every node}=[draw,circle,fill=black,minimum size=2.5pt,
		inner sep=0pt]
		\draw (0,0) node (1) [] {}
		-- ++(0:\largo) node (2) [] {}
		++(120:\largo) node (3) [] {}
		-- (1);
		\draw (2) edge[bend left = 20] (3);
		\draw (2) edge[bend right = 20] (3);
		\end{tikzpicture}
		\caption*{\tiny$G_4$}
	\end{minipage}
	\hspace{0.5cm}
	% 5 dibujo
	\begin{minipage}[b]{\porcion \textwidth}
		\centering
		\begin{tikzpicture}
		\tikzstyle{every node}=[draw,circle,fill=black,minimum size=2.5pt,
		inner sep=0pt]
		\draw (0,0) node (41) []{}
		--++(90:0.33*\largo) node (42) []{}
		++(90:0.33*\largo) node (43) []{}
		-- ++(90:0.33*\largo) node (44)[]{};
		\draw (42) edge[bend left=40] (43);
		\draw (42) edge[bend right=40] (43);
		\end{tikzpicture}
		\caption*{\tiny$G_5$}
	\end{minipage}
	\hspace{-0.5cm}
	%6 dibujo
	\begin{minipage}[b]{\porcion \textwidth}
		\centering
		\begin{tikzpicture}
		\tikzstyle{every node}=[draw,circle,fill=black,minimum size=2.5pt,
		inner sep=0pt]
		\draw (0,0) node (31) [] {}
		++(90:0.33*\largo) node (32) [] {}
		-- ++(90:0.33*\largo) node (33) [] {}
		-- ++(90:0.33*\largo) node (34)[]{};
		\draw (31) edge[bend left=40] (32);
		\draw (31) edge[bend right=40] (32);
		\end{tikzpicture}
		\caption*{\tiny$G_6$}
	\end{minipage}
	\hspace{-0.3cm}
	
	\medskip
	% 7 dibujo
	\begin{minipage}[b]{\porcion \textwidth}
		\centering
		\begin{tikzpicture}
		\tikzstyle{every node}=[draw,circle,fill=black,minimum size=2.5pt,
		inner sep=0pt]
		\draw (0,0) node (1) []{}
		++(90: 0.5*\largo) node (2)[]{}
		--++(45: 0.5*\largo) node (3)[]{};
		\draw[] (2) --++(135: 0.5*\largo) node (4)[]{};
		\draw (1) edge[bend left=30] (2);
		\draw (1) edge[bend right=30] (2);	
		\end{tikzpicture}
		\caption*{\tiny$G_{7}$}
	\end{minipage}
	% 8 dibujo
	\begin{minipage}[b]{\porcion \textwidth}
		\centering
		\begin{tikzpicture}
		\tikzstyle{every node}=[draw,circle,fill=black,minimum size=2.5pt,
		inner sep=0pt]
		\draw (0,0) node (1) []{}
		++(90: 0.5*\largo) node (2) []{}
		++(90: 0.5*\largo) node (3) []{};
		\draw (1) edge[bend left=30] (2);
		\draw (1) edge[bend right=30] (2);
		\draw (2) edge[bend left=30] (3);
		\draw (2) edge[bend right=30] (3);
		\end{tikzpicture}
		\caption*{\tiny$G_{8}$}
	\end{minipage}
	% 9 dibujo
	\begin{minipage}[b]{\porcion \textwidth}
		\centering
		\begin{tikzpicture}
		\tikzstyle{every node}=[draw,circle,fill=black,minimum size=2.5pt,
		inner sep=0pt]
		\draw (0,0) node (1) []{}
		++(90: 0.5*\largo) node (2) []{}
		--++(90: 0.5*\largo) node (3) []{};
		\draw (0.3*\largo,0) node (4) []{}
		--++(90: \largo) node (5) []{};
		\draw (1) edge[bend left=30] (2);
		\draw (1) edge[bend right=30] (2);	
		\end{tikzpicture}
		\caption*{\tiny$G_{9}$}
	\end{minipage}
	% 10 dibujo
	\begin{minipage}[b]{\porcion \textwidth}
		\centering
		\begin{tikzpicture}
		\tikzstyle{every node}=[draw,circle,fill=black,minimum size=2.5pt,
		inner sep=0pt]
		\draw (0,0) node (1) []{}
		++(90:\largo) node (2) []{};
		\draw (0.3*\largo,0) node (3) []{}
		-- ++(90: 0.5*\largo) node (4) []{}
		-- ++(90: 0.5*\largo) node (5) []{};
		\draw (1) edge[bend left=20] (2);
		\draw (1) edge[bend right=20] (2);	
		\end{tikzpicture}
		\caption*{\tiny$G_{10}$}
	\end{minipage}
	% 11 dibujo
	\begin{minipage}[b]{\porcion \textwidth}
		\centering
		\begin{tikzpicture}
		\tikzstyle{every node}=[draw,circle,fill=black,minimum size=2.5pt,
		inner sep=0pt]
		\draw (0,0) node (1) []{}
		++(90:\largo) node (2) []{};
		\draw (0.3*\largo,0) node (3) []{}
		++(90:\largo) node (4) []{};
		\draw (1) edge[bend left=20] (2);
		\draw (1) edge[bend right=20] (2);
		\draw (3) edge[bend left=20] (4);
		\draw (3) edge[bend right=20] (4);	
		\end{tikzpicture}
		\caption*{\tiny$G_{11}$}
	\end{minipage}
	% 12 dibujo
	\begin{minipage}[b]{\porcion \textwidth}
		\centering
		\begin{tikzpicture}
		\tikzstyle{every node}=[draw,circle,fill=black,minimum size=2.5pt,
		inner sep=0pt]
		\draw (0,0) node (1) []{}
		++(90:\largo) node (2) []{};
		\draw (0.3*\largo,0) node (3) []{}
		-- ++(90:\largo) node (4) []{};
		\draw (0.6*\largo,0) node (3) []{}
		-- ++(90:\largo) node (4) []{};
		\draw (1) edge[bend left=20] (2);
		\draw (1) edge[bend right=20] (2);	
		\end{tikzpicture}
		\caption*{\tiny$G_{12}$}
	\end{minipage}
\caption{Possible non-acyclic multi-graphs for $n$ odd.}
\label{fig:Nodd}
\end{figure}

We will now estimate the contribution of each of these possible non-acyclic multi-graphs.

\begin{claim}\label{claim:Bounds1Isomorphism}
	Let $n\in \mathbb{N}$, then
	$$\#\{Y\in\mathcal{Y}: G(Y) \cong G_c\}=
	\begin{cases}
		O(n^4), & \mbox{ if } c = 1,2,3,5,6,7,12;\\
		O(n^3), & \mbox{ if } c = 4,8,9,10,11.
	\end{cases}$$
\end{claim}
 \begin{proof}[\textbf{Proof}]
 	The cases $c = 1,2,5,6,7$ can be bounded by the trivial bound $O(n^4)$, and the same for the cases $c=4,8$ with the bound $O(n^3)$. The remaining cases $c=3,9,10,11,12$ require better estimates than their respective trivial bounds.\\
 	
 	First, notice that for all cases, the four edges of the multi-graph $G(\{y_{i,j},y_{i,j'},y_{r,s},y_{r,s'}\})$ are divided into two pairs: $e(y_{i,j}),e(y_{i,j'})$ of length $i$ and $e(y_{r,s}),e(y_{r,s'})$ of length $r$. The case $G_3$ is bounded by ${n \choose 3} * 2n =  O(n^4)$ because three vertices can be chosen freely to form a triangle whose edges have at most two different lengths $i,r$, then we choose a vertex $v$ for the free edge and finally we choose $v'$ such that $\dist{v-v'} = i$ or $\dist{v-v'} = r$ depending on the lengths of the edges in the triangle, therefore $v'$ has only two choices.\\
 	
 	The case $G_{12}$ is also bounded by $O(n^4)$. To show this, we distinguish between two subcases. In the first subcase, the multi-edge is formed of the associated edges of the same pair, w.l.o.g.\ $e(y_{i,j}) = e(y_{i,j'})$ (this can only happen in the case $n$ even). Then the free edges are formed of the edges $e(y_{r,s}), e(y_{r,s'})$, which have length $r$; we choose two vertices for the multi-edge and two more vertices $v_1,v_2$ (one for each of the free edges), but then the two missing vertices $v_{1}',v_{2}'$ have at most two options each, because $\dist{v-v_1} = \dist{v_2-v_{2}'} = r$. Thus this subcase is bounded by $O(n^4)$. The second subcase is when $e(y_{i,j}) \neq e(y_{i,j'})$ and $e(y_{r,s}) \neq e(y_{r,s'})$. Then w.l.o.g.\ the multi-edge is formed of the $e(y_{i,j}) = e(y_{r,s})$ then $i=r$, thus all edges have the same length; we choose two vertices $v,v'$ for the multi-edge and two more vertices $v_1,v_2$ (one for each of the free edges). The missing vertices $v_{1}',v_{2}'$ have at most two choices each because $\dist{v_1-v_{1}'} = \dist{v_2-v_{2}'} = \dist{v-v'}$, which gives again a $O(n^4)$ bound.\\
 	
 	For $G_{9}$, if we are in the case $n$ odd, then the multi-edge is formed of edges of different groups, w.l.o.g.\ $e(y_{i,j}) = e(y_{r,s})$ and $i=r$. Therefore the edge attached to the multi-edge is uniquely defined because its length is determined, and the isolated edge is almost uniquely defined once one of the end points is chosen, because the other end has at most two choices. Overall, this gives the $O(n^3)$ bound. In the case $n$ even, it can happen that w.l.o.g.\ $e(y_{i,j}) = e(y_{i,j'})$ but this can only happen when $i = n/2$. Then the multi-edge is uniquely defined by choosing one end, the isolated edge is defined by choosing two end points, and the last edge has at most four options since its length is already determined by the length of the isolated edge. This gives again a $O(n^3)$ bound.\\
 	
 	For $G_{10}$, in the case $n$ odd we can assume as before $e(y_{i,j}) = e(y_{r,s})$. Then $i=r$, and the multi-edge is determined by choosing two vertices and the remaining two edges are uniquely defined by the central vertex. This yields the bound $O(n^3)$. In the other case, w.l.o.g.\ $e(y_{i,j}) = e(y_{i,j'})$, and $i=n/2$. The multi-edge can be defined by choosing only one vertex, and the isolated path can be defined by choosing two vertices for one edge,  while the remaining edge will have at most two options. This yields again a $O(n^3)$ bound.\\
 	
 	For $G_{11}$, if $e(y_{i,j}) = e(y_{r,s})$, then all edges have the same length $i=r$, we can choose two vertices for the first multi-edge and one vertices for the second multi-edge, while the remaining vertex has at most two options. This yields a $O(n^3)$ bound. In the case when $e(y_{i,j}) = e(y_{i,j'})$ then $e(y_{r,s}) = e(y_{r,s'})$ and  $i=r=n/2$. In this case we can choose two vertices (one for each multi-edge), and the remaining two vertices are automatically determined. This yields a $O(n^2) = O(n^3)$ bound. Thus we have established Claim \ref{claim:Bounds1Isomorphism}.
 \end{proof}

 \noindent We continue with the proof of Claim \ref{claim:SecondMoments} in the case when $n$ is odd. We observe that
\begin{equation*}\label{eq:EyyyyOdd}
\E{y_{i,j}y_{i,j'}y_{r,s}y_{r,s'}}
	=
	\prob{y_{i,j}y_{i,j'}y_{r,s}y_{r,s'} = 1}
	= \prob{\left\{\mathbf{b}:~ \Tb(i,j)=\Tb(i,j')=\Tb(r,s)=\Tb(r,s')=0 \right\} },
\end{equation*}
and thus for $Y = \{y_{i,j},y_{i,j'},y_{r,s},y_{r,s'}\} \in  \mathcal{Y}$, we have that
\begin{equation}\label{eq:EY}
\E{y_{i,j}y_{i,j'}y_{r,s}y_{r,s'}}
=
\begin{cases}
\frac{1}{n^3}, &\mbox{if } G(Y )\cong G_{1,2,3,5,6,7,9,10,12};\\
\frac{1}{n^2}, &\mbox{if } G(Y) \cong G_{4,8,11}.
\end{cases}
\end{equation}
The last equation, combined with Claim~\ref{claim:Bounds1Isomorphism}, implies that
\begin{align*}
	\sum_{\sumOne}
	\cov{y_{i,j}y_{i,j'},y_{r,s}y_{r,s'}}
	&\leq
	\sum_{\sumOne}
	\E{y_{i,j}y_{i,j'}y_{r,s}y_{r,s'}}
	\\&\leq
	7 \cdot O(n^4)\frac{1}{n^3}
	+
	3 \cdot O(n^3)\frac{1}{n^2}
	+
	2 \cdot O(n^3)\frac{1}{n^3}
	\\& =
	O(n).
\end{align*}
Using the previous inequality and \eqref{eq:UpperBoundFirstSumVZ1Odd} we get that
\begin{equation}\label{eq:VZ1OddO(n)}
	\var{\Z_1}
	=
	\sum_{\substack{1 \leq i\leq \nmed \\0 \leq j < j' \leq n-1}}
	\var{y_{i,j}y_{i,j'}}
	+
	\sum_{\sumOne}
	\cov{y_{i,j}y_{i,j'},y_{r,s}y_{r,s'}}
	=
	O(n)+ O(n)
	= O(n).
\end{equation}
This completes the proof of Claim \ref{claim:SecondMoments} in the case when $n$ is odd.\\

\indent Case 2: $n$ even.
We estimate the variances of $\Z_0$ and $\Z_1$. For $n$ even, we can write $\Z_0$ as
$$
	\Z_0
	=
	\sum_{\substack{1\leq i < \med\\0 < j \leq n-1}} y_{i,j}
	+
	2\sum_{j=0}^{\med -1}
	 y_{\med,j},
$$
where all variables involved in the sums are mutually independent (see Remark~\ref{rmk:EqualEdges}). Thus
$$
	\var{\Z_0}
	=
	\sum_{\substack{1\leq i < \med\\0 \leq j \leq n-1}} \var{y_{i,j}}
	+
	4\sum_{j=0}^{\med - 1}
	\var{y_{\med,j}}.
$$
Using \eqref{eq:Ey}, we deduce that
\begin{equation}\label{eq:VZ0EvenO(n)}
	\var{\Z_0}
=
\left(\med -1 \right)
n
\left(\frac{1}{n} - \frac{1}{n^2}\right)
+
4\med \left( \frac{1}{n} - \frac{1}{n^2} \right)
=
O(n),
\end{equation}
for all $n$ even.
 By Remark~\ref{rmk:EqualEdges}, we can write $\Z_1$ as
$$
\Z_1
=
\sum_{\substack{1 \leq i \leq \med \\0 \leq j < j' \leq n-1\\j\not\equiv j' \Mod{n/2}}}
y_{i,j}y_{i,j'}	
+
\sum_{s = 0}^{\frac{n}{2} - 1}
y_{n/2,s}.
$$
Therefore
\begin{equation}\label{eq:VZ1ExpansionEven}
\begin{split}
\var{\Z_1}
&=
\sum_{\substack{1 \leq i \leq \med \\0 \leq j < j' \leq n-1\\j\not\equiv j' \Mod{n/2}}}
\var{y_{i,j}y_{i,j'}}
+
\sum_{s = 1}^{\frac{n}{2} - 1}
\var{y_{n/2,s}}
+
\sum_{\sumOneEven}
\cov{y_{i,j}y_{i,j'}, y_{r,s}y_{r,s'}}\\
\\&+
2\sum_{\substack{1\leq u \leq \med \\ 0 \leq v < v' \leq n-1 \\ v \not\equiv_\med v' \\
0\leq w \leq \med -1}}
\cov{y_{u,v}y_{u,v'}, y_{\med,w}}
+
\underbrace{\sum_{\substack{ 0\leq w,w' \leq \med -1\\ w\neq w'}}
\cov{y_{\med,
	w}, y_{\med,w'}}}_{= \, 0 \text{ (\tiny by Remark~\ref{rmk:EqualEdges})}}.
\end{split}
\end{equation}
We divide the analysis into three parts: the first two sums, the third sum, and the fourth sum. Using Remark~\ref{rmk:EqualEdges}, we write the first two sums in \eqref{eq:VZ1ExpansionEven} as
\begin{equation}\label{eq:VZ1EvenFirstTwoSums}
\sum_{\substack{1 \leq i \leq \med \\0 \leq j < j' \leq n-1\\j\not\equiv j' \Mod{n/2}}}
\var{y_{i,j}}\var{y_{i,j'}}
+
\sum_{s = 1}^{\frac{n}{2} - 1}
\var{y_{n/2,s}}
\stackrel{\mbox{\tiny \eqref{eq:Ey} }}{\leq}
n \cdot n^2 \left(\frac{1}{n}-\frac{1}{n^2}\right)^2
+
n \left(\frac{1}{n}-\frac{1}{n^2}\right)
=
O(n).
\end{equation}
The third sum in \eqref{eq:VZ1ExpansionEven} can be bounded above in the same way as in the odd case: the associated graphs of variables $y_{i,j},y_{i,j'},y_{r,s},y_{r,s'}$ with non-zero covariance in the third sum, are isomorphic to one of the graphs in Figure~\ref{fig:Nodd}. Thus we can use Claim~\ref{claim:Bounds1Isomorphism} and \eqref{eq:EY} to obtain
\begin{equation}\label{eq:CovSum2}
\sum_{\sumOneEven}
\cov{y_{i,j}y_{i,j'},y_{r,s}y_{r,s'}}
=
O(n).
\end{equation}
In the fourth sum in \eqref{eq:VZ1ExpansionEven}, the variables with non-zero covariance have an associated multi-graph which is isomorphic to one of the following multi-graphs.
\begin{figure}[h!]
\centering
% Primer dibujo
\begin{minipage}{0.15 \textwidth}
\centering
	\begin{tikzpicture}
	\tikzstyle{every node}=[draw,circle,fill=black,minimum size=2.5pt,
	inner sep=0pt]
	\draw (0,0) node (11) [] {}
	-- ++(0:1.0cm) node (12) [] {}
	-- ++(120:1.0cm) node (13) [] {}
	-- (11);
	\end{tikzpicture}
	\caption*{\tiny $G_{13}$}
\end{minipage}
%Segundo dibujo
\begin{minipage}{0.15 \textwidth}
	\centering
	\begin{tikzpicture}
	\tikzstyle{every node}=[draw,circle,fill=black,minimum size=2.5pt,
	inner sep=0pt]
	\draw (3,0) node (21) []{}
	++(90:1.0cm) node (22) []{};
	% Aristas curvas
	\draw (21) edge[bend left=20] (22);
	\draw (21) edge[bend right=20] (22);
	% Resto del segundo dibujo
	\draw (3.5,0) node (23) []{}
	-- ++(90:1.0cm) node (24) []{};
	\end{tikzpicture}
	\caption*{\tiny$G_{14}$}
\end{minipage}
% Tercer dibujo
\begin{minipage}{0.15 \textwidth}
	\centering
	\begin{tikzpicture}
	\tikzstyle{every node}=[draw,circle,fill=black,minimum size=2.5pt,
	inner sep=0pt]
	\draw (0,0) node (31) []{}
	++(90:0.5) node (32) []{}
	-- ++(90:0.5) node (33) []{};
	\draw (31) edge[bend left=30] (32);
	\draw (31) edge[bend right=30] (32);
	\end{tikzpicture}
	\caption*{\tiny $G_{15}$}
\end{minipage}

\end{figure}

 \noindent Let $\mathcal{X} := \left\{\{y_{u,v},y_{u,v'},y_{\med,w}\}: 1\leq u \leq \med; 0 \leq v < v' \leq n-1; v\not\equiv_{\med}v';0\leq w \leq \med -1\right\}$. In the same way as  Claim~\ref{claim:Bounds1Isomorphism}, we can prove that $$\#\left\{X\in\mathcal{X}: G(X)\cong G_{c}\right\} = O(n^3), \quad c=13,14,15.$$ As in \eqref{eq:EY}, we can prove that $\E{y_{u,v}y_{u,v'}y_{\med,w}} = \frac{1}{n^2}$ for all $X = \{y_{u,v},y_{u,v'},y_{\med,w}\}\in \mathcal{X}$. Thus
\begin{equation}\label{eq:CovSum3}
\sum_{\substack{1\leq u \leq \med \\ 0 \leq v < v' \leq n-1 \\ 0\leq w \leq \med -1}}
\cov{y_{u,v}y_{u,v'},y_{\med,w}}
\leq
3\cdot O(n^3)\frac{1}{n^2}
=
O(n).
\end{equation}
Plugging \eqref{eq:VZ1EvenFirstTwoSums},\eqref{eq:CovSum2},\eqref{eq:CovSum3} into \eqref{eq:VZ1ExpansionEven}  finally yields
\begin{equation}\label{eq:VZ1EvenO(n)}
	\var{\Z_1}
	=
	O(n) + O(n) + 2\cdot O(n)
	=
	O(n),
\end{equation}
for all $n$ even. Equations \eqref{eq:VZ0OddO(n)},\eqref{eq:VZ1OddO(n)},\eqref{eq:VZ0EvenO(n)} and \eqref{eq:VZ1EvenO(n)} together yield  Claim~\ref{claim:SecondMoments} in the case when $n$ is even. Thus we have fully established Claim~\ref{claim:SecondMoments}.
\end{proof}

\subsection{$\mathcal{E}_{\mathrm{\small zero}} (\frac{1}{2}-\varepsilon)$ has high probability}
%{$\mathcal{E}^{0.49 (1-\varepsilon)}_{\mbox{\tiny zero}}$ holds w.h.p. }
 Using Chebyshev's inequality, we obtain that
 \begin{align*}
	 &\prob{|\Z_0-\E{\Z_0}|\geq \lambda_0}
	 \leq
	 \frac{\var{\Z_0}}{\lambda_{0}^2};
	 &
	 \prob{|\Z_1-\E{\Z_1}|\geq \lambda_1}
	 \leq
	 \frac{\var{\Z_1}}{\lambda_{1}^2},
 \end{align*}
  for every $\lambda_0,\lambda_1 >0$.
 In particular, this implies that
 \begin{align*}
 &\prob{\Z_0 < \E{\Z_0} - \lambda_0}
 \leq
 \frac{\var{\Z_0}}{\lambda_{0}^2};
 &
 \prob{\Z_1>\E{\Z_1} + \lambda_1}
 \leq
 \frac{\var{\Z_1}}{\lambda_{1}^2}.
 \end{align*}
Let $\varepsilon \in (0,1)$ be the constant from the statement of Lemma~\ref{lemma:ConcentrationOfD}, and set $\nu=\varepsilon n/8$. Choosing $\lambda_0 = \lambda_1 = \nu$ and using Claims~\ref{claim:FirstMoments} and \ref{claim:SecondMoments} we get that
 \begin{align*}
&\prob{\Z_0 < \E{\Z_0} - \nu}
\leq
\frac{\var{\Z_0}}{\nu^2}
=\frac{O(n)}{n^2}
=
O\left(\frac{1}{n}\right)
;
\\&
\prob{\Z_1>\E{\Z_1} +\nu}
\leq
\frac{\var{\Z_1}}{\nu^2}
=
\frac{O(n)}{n^2}
=
O\left(\frac{1}{n}\right).
\end{align*}
By Claim~\ref{claim:FirstMoments} we have 
$$
\delta := (\frac{1}{2}-\varepsilon) \nmed \leq \E{\Z_0-\Z_1} - 2 \nu
$$
%for all sufficiently large $n$.
for $n$ sufficiently large.
Thus, using Claim \ref{claim:ConcentrationReduction} we can conclude that
\begin{align*}
%	\prob{\left(\mathcal{E}^{0.49  (1-\varepsilon)}_{\mbox{\tiny zero}}\right)^c}
\prob{\mathcal{E}_{\mathrm{\small zero}}^c (\frac{1}{2}-\varepsilon) } & =
	\prob{\left\{  \mathbf{b} \in \mathcal{M}_n:~\D(\mathbf{b}) < (\frac{1}{2}-\varepsilon) \nmed \right\}} \\
	& \leq
	\underbrace{\mathbb{P}\Big[\Z_0 < \E{\Z_0} - \nu \Big]}_{\mbox{\tiny$= O\left(\frac{1}{n}\right)$}}
	+
	\underbrace{\mathbb{P}\Big[\Z_1>\E{\Z_1} + \nu \Big]}_{\mbox{\tiny$= O\left(\frac{1}{n}\right)$}}
	\\&=
	O\left(\frac{1}{n}\right).
\end{align*}
This concludes the proof of Lemma~\ref{lemma:ConcentrationOfD}.
%****************************************************
\section{Connections with chromatic polynomials of circulant graphs}\label{sec:ChromaticConnections}
As we have already seen in the proof of Claim~\ref{claim:SubsetIndependent}, the multi-graph associated with the variables in row $i \neq \frac{n}{2}$ of $\Tb$ is the circulant graph $C_n(i)$, and the same holds for the variables in row $n/2$ if we consider the associated graph and not the associated multi-graph. Furthermore, we can express the probability of synchronization of circular automata in terms of chromatic polynomials of circulant graphs: this is a consequence of the close connection of the moments of  $\D(\b{})$ to chromatic polynomials of circulant graphs. We formalize this in the following results.\\
\begin{defi}
	The \textit{circulant graph} $C_n(i_1,i_2,\ldots,i_k)$ is a graph with vertex set $\zn$ where two vertices $r,s$ are adjacent if $\dist{r-s} \in \{ i_1,i_2,\ldots,i_k\}.$
\end{defi}
\begin{defi} Let $G$ be a graph with vertex set $\{0,1,\ldots,n-1\}$. The \textit{chromatic polynomial} $P(G;x): \mathbb{N} \rightarrow \mathbb{N}$ of $G$ is defined by
	$$
	P(G;x) := \#\{ \mathbf{b} \in \{0,\ldots,x-1\}^n: ~\mathbf{b} \text{ is a proper coloring of } G\}.
	$$
\end{defi}
\begin{remark}
	Let $G$ be of order $n$. Then
	$P(G;x) = \sum_{j=1}^{n} \lambda_j x^j,$ where $\lambda_j \in \mathbb{Z}$ (see, for instance, \cite{DongBook}).
\end{remark}
\begin{claim}\label{claim:VD} Let $\D(\mathbf{b})$ and $\b{}=(b_0,b_1,\ldots,b_{n-1}) \in \mathcal{M}_n$ be as in Lemma~\ref{lemma:ConcentrationOfD}. Then
	\begin{align*}
	&\E{\D}= \nmed - \sum_{i=1}^{\nmed} \frac{P_i(n)}{n^n}
	\end{align*}
	and
	\begin{align*}
	& \var{\D} =
	\sum_{i=1}^{n} \left( \frac{P_i(n)}{n^n} - \frac{P_{i}^2(n)}{n^{2n}} \right)
	+
	2\sum_{1\leq i<j\leq\nmed} \left(\frac{P_{i,j}(n)}{n^n} - \frac{P_i(n)P_j(n)}{n^{2n}}\right),
	\end{align*}
	where $P_i$ is the chromatic polynomial of the circulant graph $C_n(i)$ and $P_{i,j}$ is the chromatic polynomial of the circulant graph $C_n(i,j)$.
\end{claim}
\begin{remark}\label{rmk:PiExplicitPijNPHard}
	$\bullet$ It is easy to derive that
	$ P_i(x) = \left((x-1)^{l_i} + (-1)^{l_i}(x-1)\right)^\frac{n}{l_i} $
	where $l_i=\frac{n}{\gcd (n,i)}$, because $C_n(i)$ is a collection of $\gcd(n,i)$ many disjoint cycles of length $\frac{n}{\gcd (n,i)}$  \cite{BoeschConnectivitiesCirculants}. With this explicit expression, an easy corollary of Claim \ref{claim:VD} is the estimate $\E{\D} \sim (1-e^{-1}) \nmed$.
	\\$\bullet$ We could not find an explicit expression for $P_{i,j}$. The calculation of the chromatic number of  circulant graphs with an arbitrary number of parameters is an NP-Hard problem \cite{CODENOTTICirculantGraphsChroPolyNPHard}. This implies that the calculation of chromatic polynomials of circulant graphs is also NP-Hard since \newline $\chi(G) = \text{argmin}_{w\in\mathbb{N}} P(G;w) > 0$ -- we believe that our unfruitful attempts to estimate $\var{\D}$ are connected to this. To circumvent these issues, the variables $\Z_0$ and $\Z_1$ in Section~\ref{sec:ProofLemmaZeros} were introduced.
\end{remark}
\begin{proof}[\textbf{Proof of Claim~\ref{claim:VD}}]
	Let us recall that $\D(\mathbf{b}) = \sum_{i=1}^{\nmed} \D_i(\mathbf{b})$, where 		$$
	\D_i(\mathbf{b})
	:=
	\begin{cases}
	1, &\mbox{ if there exist } \, k,l \in \zn
	\mbox{ such that } \dist{k-l} = i \mbox{ and } \dist{b_k-b_l} = 0.
	\\
	0, &\mbox{ otherwise, }
	\end{cases}
	$$
	Then $\D_i(\mathbf{b}) = 1 - x_{i}(\mathbf{b})$, where
	$$
	x_i(\b{}) := \prod_{j=0}^{n} \left(1 - \mathds{1}\{\dist{b_j - b_{(j+i)_n}}=0\}\right).
	$$
	We observe that	$x_i(\mathbf{b}) = 1$ if and only if every two numbers $r,s \in \zn$ at cyclic distance $i$ have different images under $\b{}$ and $x_i(\mathbf{b}) = 0$ otherwise. If we consider $\b{}$ as a random coloring of $C_n(i)$, then $x_i(\b{}) = 1$ if and only if $C_n(i)$ is properly colored by $\b{}$. Thus
	$$\E{x_i} = \prob{\{\mathbf{b}:~x_i(\b{}) = 1 \}} = \frac{P_i(n)}{n^n}.$$
	In a similar way
	$$
	\E{x_i x_j} = \prob{\{\mathbf{b}:~x_i(\b{}) x_j(\b{}) = 1\}}
	=
	\frac{P_{i,j}(n)}{n^n}.
	$$
	Therefore
	$$
	\E{\D} = \sum_{i=1}^{\nmed} \E{\D_i}
	=
	\sum_{i=1}^{\nmed} \left(1 - \E{x_i}\right)
	=
	\nmed - \sum_{i=1}^{\nmed} \frac{P_i(n)}{n^n},
	$$
	as well as
	$$
	\var{\D_i}
	=
	\E{\D_{i}^2}-\E{\D_i}^2
	=
	\left(1-\frac{P_{i}(n)}{n^{n}}\right)
	-
	\left(1-\frac{P_{i}(n)}{n^{n}}\right)^2
	=
	\frac{P_i(n)}{n^n}
	- \frac{P_{i}^2(n)}{n^{2n}}\\
	$$
	and
	\begin{align*}
	\cov{\D_i,\D_j}
	=
	\E{\D_i\D_j}-\E{\D_i}\E{\D_j}
	&=
	\E{(1-x_i)(1-x_j)}- \E{1-x_i}\E{1-x_j}
	\\&=
	\E{x_ix_j}-\E{x_i}\E{x_j}
	\\&=
	\frac{P_{i,j}(n)}{n^n} - \frac{P_i(n)P_j(n)}{n^{2n}}.
	\end{align*}
	Plugging the two previous equations into
	\begin{align*}
	\var{\D}
	=
	\sum_{i=1}^{\nmed} \var{\D_i}
	+
	2\sum_{1\leq i<j\leq\nmed}
	\cov{\D_i,\D_j}
	\end{align*}
	yields Claim~\ref{claim:VD}.
\end{proof}
\noindent We get the following relation between chromatic polynomials of circulant graphs and  synchronization of circular automata.
The number $\frac{1}{2} - e^{-1}$ in the statement of Theorem~\ref{thm:SyncChromPoly} has the approximate value $0.13$.

\begin{thm}\label{thm:SyncChromPoly}
	Let $\mathcal{A}_n (\mathbf{b})$ be a circulant graph as introduced in Section \ref{sec:MainResult}. 
	% with $n$ states, where $\zn := \{0,1,\ldots,n-1\}$, where $a$ is the circular permutation $a(i) = (i+1) \mod n$,
	%	and where $ \mathbf{b} := (b_{0},...,b_{n-1})$ is a random vector such that $b_{0},...,b_{n-1}$ are i.i.d.\ random variables which are uniformly distributed over $\zn$.
	Let $\varepsilon \in (0,\frac{1}{2} - e^{-1})$, then there exist $n_{\epsilon} \in \mathbb{N}$ such that for all $n  \geq  n_{\epsilon}$ it holds that
	$$
	\prob{\{\mathbf{b} \in \mathcal{M}_n:~\mathcal{A}_n (\mathbf{b})\mbox{ synchronizes} \}}
	\geq
	1 - \nmed \exp\left\{ -\frac{1}{2n}\left(\varepsilon \nmed - 1\right)^2\right\}
	-
	\frac{\var{\D}}{\left(\varepsilon\nmed-1\right)^2},
	$$
	where $\var{\D}$ is as given in Claim~\ref{claim:VD}.
\end{thm}

\begin{proof}[\textbf{Proof}]
	%	[\textbf{Proof of Theorem~\ref{thm:SyncChromPoly}}]
	By \eqref{eq:Lamdavarepsilon},\eqref{eq:ExplicitBoundRows} we know that
	\begin{equation}\label{eq:E^crowExplicit}
	\prob{\cRUstar}
	\leq
	\nmed \exp\left\{-\frac{1}{2n}\left( \varepsilon \nmed - 1 \right)^2\right\},
	\end{equation}
	for all $\varepsilon>0$ and $n$ large enough, where $\alpha^\star = 1-e^{-1}-\varepsilon$. Using the expression for $P_i$ in Remark~\ref{rmk:PiExplicitPijNPHard} together with the inequality $1-x \leq e^{-x},~x\in\mathbb{R}$, we bound $P_i(n)/n^n$ from above
	\begin{align*}
	\frac{\P{i}}{n^n}
	&=
	\left(\frac{n-1}{n}\right)^n\left(1+\frac{(-1)^{\ell_i}}{(n-1)^{\ell_i-1}}\right)^{\frac{n}{\ell_i}}
	\\
	&\leq
	e^{-1}
	\left(1+\frac{1}{(n-1)^{\ell_i-1}}\right)^{\frac{n}{\ell_i}}
	\\
	&\leq
	e^{-1}
	e^{\frac{n}{\ell_{i}} \cdot	 \frac{1}{(n-1)^{\ell_{i}-1}}}
	\\
	&=
	\exp
	\left\{
	-1
	+ \frac{n}{\ell_{i}\cdot(n-1)^{\ell_{i}-1}}
	\right\}		
	\end{align*}
	and thus
	\begin{equation}\label{eq:UpperBoundChromaticPoly}
	\frac{\P{i}}{n^n}
	\leq
	\begin{cases}
	\exp 
	\left\{
	-1 
	+
	\frac{1}{2} \cdot
	\left(\frac{n}{n-1}\right)
	\right\}
	, \mbox{ if } i = \frac{n}{2} \mbox{ i.e. } \ell_{i} = 2;	
	\\
	\\
	\exp 
	\left\{
	-1
	+ 
	\frac{n}{3(n-1)^2}
	\right\}
	, \mbox{ if } i \neq \frac{n}{2} \mbox{ i.e. } \ell_{i} \geq 3.
	\end{cases}	
	\end{equation}	
	Using Equation~\ref{eq:UpperBoundChromaticPoly} and the equation $\E{\D}= \nmed - \sum_{i=1}^{\nmed} \frac{P_i(n)}{n^n}$ from Claim~\ref{claim:VD} we get that
	$$
	\E{\D}
	\geq
	\nmed
	\left(
	1 - \exp\left\{ \frac{n}{3(n-1)^2} - 1\right\}
	\right)
	-1 = \eta_\star .
	$$
	By Chebyshev's inequality and elementary manipulations, we get that
	$$
	\prob{ \{ \mathbf{b} \in \mathcal{M}_n:~\D(\mathbf{b}) < \eta_\star - \lambda  \}}
	\leq \frac{\var{\D}}{\lambda^2},
	$$
	for all $\lambda>0$. Let $\varepsilon >0$. Setting $\lambda = \lambda_{\varepsilon}'(n)= \eta_\star - \nmed (1- e^{-1} -\varepsilon) + 1$  and noting that $\lambda >0$ for $n$ large enough, we get that
	\begin{equation}\label{eq:E^czeroRecall}
	\prob{\left(\mathcal{E}_{\text{\tiny zero}}^{\tilde{\beta}}\right)^c}
	=
	\prob{ \left\{ \mathbf{b} \in \mathcal{M}_n:~\D(\mathbf{b}) < \nmed \left(1-e^{-1}-\varepsilon \right) - 1 \right\}}
	\leq \frac{\var{\D}}{\left(\lambda_{\varepsilon}'(n)\right)^2}
	\leq
	\frac{\var{\D}}{\left(\nmed \varepsilon -1\right)^2}
	\end{equation}
	for $n$ sufficiently large, where $\tilde{\beta} =  1-e^{-1}-\varepsilon - \frac{1}{\nmed}$.
	% 1/\nmed$.
	Using the previous inequalities, we conclude that
	\begin{align}
	\prob{ \left\{  \mathbf{b} \in \mathcal{M}_n:~\auto(\b{}) \mbox{ synchronizes}  \right\}}
	&\label{eq:SynchroReduction2}\stackrel{\text{ \eqref{eq:SynchroMatrixReduction}}}{\geq}
	1- \prob{\cRUstar} - \prob{\mathcal{E}_{\mathrm{\small zero}}^c (\tilde{\beta})}
	%	\left(\mathcal{E}_{\text{\tiny zero}}^{\tilde{\beta}}\right)^c}
	\\&\geq
	1 - \nmed \exp\left\{ -\frac{1}{2n}\left(\varepsilon \nmed - 1\right)^2\right\}
	-
	\frac{\var{\D}}{\left(\varepsilon\nmed-1\right)^2}
	\end{align}
	for $n$ large enough where the relations $\alpha^\star,\tilde{\beta}>0$ and $\alpha^\star + \tilde{\beta} >1$ are valid when $\varepsilon \in (0,\frac{1}{2} - e^{-1})$ and $n$ is large enough.
\end{proof}
\noindent Actually, we formulate the following conjecture:
\begin{conj}
	$\var{\D} = O(n).$  \end{conj}

%\noindent This conjecture can be reduced 
\noindent To prove this conjecture it is sufficient to prove that there is $g:\mathbb{N} \rightarrow \mathbb{R}$ such that $|\frac{P_{i,j}(n)}{n^n} - \frac{P_i(n)P_j(n)}{n^{2n}} | \leq g(n) = O(1/n)$ for all $i,j$. From \eqref{eq:UpperBoundChromaticPoly} we see that $0 \leq P_i(n)/n^n \leq f(n) = O(1)$ for all $i$, therefore the first part of the sum of $\var{\D}$ given in Claim~\ref{claim:VD} is $|\sum_{i=1}^{n} \left( \frac{P_i(n)}{n^n} - \frac{P_{i}^2(n)}{n^{2n}} \right)| \leq n f(n) = O(n)$. The second part of the sum $\sum_{1\leq i<j\leq\nmed} \left(\frac{P_{i,j}(n)}{n^n} - \frac{P_i(n)P_j(n)}{n^{2n}}\right)$ has a quadratic number of elements of the form $\frac{P_{i,j}(n)}{n^n} - \frac{P_i(n)P_j(n)}{n^{2n}}$, and it can be bounded by $O(n^2)g(n) = O(n) $ if the assumption $|\frac{P_{i,j}(n)}{n^n} - \frac{P_i(n)P_j(n)}{n^{2n}}| \leq g(n) = O(1/n)$ for all $i,j$ is true, making   $\var{\D} = O(n) + O(n) = O(n).$ In particular, a positive answer to this chromatic-polynomial question would give an alternative proof of Theorem~\ref{thm:main}.

%****************************************************
\section{Future work}\label{sec:FutureWork}
Let $\auto(\mathbf{a},\b{})$ be an automaton where $\mathbf{a}:\zn \rightarrow \zn$ is fixed and  $\b{} \in \mathcal{M}_n$. These are natural lines of research to extend/improve the results in this paper:\\
$\bullet$ We want to explore in more detail the strengths and limitations in the ideas presented in this paper. For example, we think that these ideas can extend  Theorem~\ref{thm:main} to the case where $\mathbf{a}:\zn \rightarrow \zn$ is in the form of a finite number of pairwise disjoint cycles of almost-equal length. We also think that (probabilistic) upper bounds for the length of the synchronizing minimal words can be given with our techniques, in the spirit of the results of \cite{nicaud2014fast}.\\
$\bullet$ Theorem~\ref{thm:Perrin}
has a decay rate in $\Theta\left(\frac{\sqrt{p}}{e^p}\right)$. We believe that this can be extended in a weaker form to the case of circular automata of composite order:

\begin{conj}
%	There is some $\varepsilon>0$ such that 
	$$\prob{\left\{ \mathbf{b} \in \mathcal{M}_n:~ \auto(\b{}) \text{ synchronizes} \right\}} =
%	1 - O(e^{-\varepsilon n})$$
	1 - O(\alpha^{n}),$$	for some $0 < \alpha < 1$, as $n \to \infty$.
\end{conj}

%****************************************************
\section*{Acknowledgments}
CA acknowledges the financial support of the Austrian Science Fund (FWF), projects F-5512, I-3466 and Y-901. DD, AG and AR acknowledge the financial support of the FWF project P29355-N35. AR acknowledges also the partial support of the FWF project P25510-N26. We want to thank two anonymous referees who read our paper very carefully, and whose suggestions greatly helped us to improve the presentation of this paper.

% References
\bibliographystyle{apalike}
\bibliography{References}

\begin{thebibliography}{}

\bibitem[Ananichev et~al., 2010]{VolkovExtremeAutomata}
Ananichev, D.~S., Gusev, V.~V., and Volkov, M.~V. (2010).
\newblock Slowly synchronizing automata and digraphs.
\newblock {\em CoRR}, abs/1005.0129.

\bibitem[B{\'e}al et~al., 2011]{beal2011}
B{\'e}al, M.-P., Berlinkov, M.~V., and Perrin, D. (2011).
\newblock A quadratic upper bound on the size of a synchronizing word in
  one-cluster automata.
\newblock {\em International Journal of Foundations of Computer Science},
  22(02):277--288.

\bibitem[Berlinkov, 2016]{Berlikov2013Synchronization}
Berlinkov, M.~V. (2016).
\newblock On the probability of being synchronizable.
\newblock In {\em Algorithms and discrete applied mathematics}, volume 9602 of
  {\em Lecture Notes in Comput. Sci.}, pages 73--84. Springer, [Cham].

\bibitem[Berlinkov and Nicaud, 2018]{Nicaud2018AlmostGroupAutomata}
Berlinkov, M.~V. and Nicaud, C. (2018).
\newblock Synchronizing random almost-group automata.
\newblock In {\em International Conference on Implementation and Application of
  Automata}, pages 84--96. Springer.

\bibitem[Boesch and Tindell, 1984]{BoeschConnectivitiesCirculants}
Boesch, F. and Tindell, R. (1984).
\newblock Circulants and their connectivities.
\newblock {\em J. Graph Theory}, 8(4):487--499.

\bibitem[Cerny, 1964]{cerny1964poznamka}
Cerny, J. (1964).
\newblock Poznamka k homogenym eksperimentom s konechnymi automatami.
\newblock {\em Math.-Fyz. Cas}, 14:208--215.

\bibitem[Codenotti et~al., 1998]{CODENOTTICirculantGraphsChroPolyNPHard}
Codenotti, B., Gerace, I., and Vigna, S. (1998).
\newblock Hardness results and spectral techniques for combinatorial problems
  on circulant graphs.
\newblock {\em Linear Algebra and its Applications}, 285(1):123 -- 142.

\bibitem[Dubuc, 1998]{dubuc1998}
Dubuc, L. (1998).
\newblock Sur les automates circulaires et la conjecture de \v{C}ern{\`y}.
\newblock {\em RAIRO-Theoretical Informatics and Applications}, 32(1-3):21--34.

\bibitem[Fengming et~al., 2005]{DongBook}
Fengming, D., Khee-meng, K., et~al. (2005).
\newblock {\em Chromatic polynomials and chromaticity of graphs}.
\newblock World Scientific.

\bibitem[McDiarmid, 1989]{mcdiarmid1989method}
McDiarmid, C. (1989).
\newblock On the method of bounded differences.
\newblock {\em Surveys in combinatorics}, 141(1):148--188.

\bibitem[Nicaud, 2019]{nicaud2014fast}
Nicaud, C. (2019).
\newblock The \v{C}ern\'{y} conjecture holds with high probability.
\newblock {\em J. Autom. Lang. Comb.}, 24(2-4):343--365.

\bibitem[Perrin, 1977]{perrin1977codes}
Perrin, D. (1977).
\newblock Codes asynchrones.
\newblock {\em Bulletin de la Soci{\'e}t{\'e} math{\'e}matique de France},
  105:385--404.

\bibitem[Pin, 1978]{MR520853}
Pin, J.-E. (1978).
\newblock Sur un cas particulier de la conjecture de {C}erny.
\newblock In {\em Automata, languages and programming ({F}ifth {I}nternat.
  {C}olloq., {U}dine, 1978)}, volume~62 of {\em Lecture Notes in Comput. Sci.},
  pages 345--352. Springer, Berlin-New York.

\bibitem[Pin, 1983]{pin1983CernyBound}
Pin, J.-E. (1983).
\newblock On two combinatorial problems arising from automata theory.
\newblock In {\em North-Holland Mathematics Studies}, volume~75, pages
  535--548. Elsevier.

\bibitem[Shitov, 2019]{Shitov2019}
Shitov, Y. (2019).
\newblock An improvement to a recent upper bound for synchronizing words of
  finite automata.
\newblock {\em Journal of Automata, Languages and Combinatorics}, 24:367--373.

\bibitem[Skvortsov and Zaks, 2010]{skvortsov20100LargeAlphabetSynchronization}
Skvortsov, E. and Zaks, Y. (2010).
\newblock Synchronizing random automata.
\newblock {\em Discrete Mathematics and Theoretical Computer Science},
  12(4):95--108.

\bibitem[Szyku{\l}a, 2017]{szykula2017improvingCernyBound}
Szyku{\l}a, M. (2017).
\newblock Improving the upper bound on the length of the shortest reset words.
\newblock {\em arXiv preprint arXiv:1702.05455}.

\bibitem[Volkov, 2008]{volkov2008synchronizing}
Volkov, M.~V. (2008).
\newblock Synchronizing automata and the \v{C}ern{\`y} conjecture.
\newblock In {\em International Conference on Language and Automata Theory and
  Applications}, pages 11--27. Springer.

\end{thebibliography}

\end{document}